\documentclass[11pt,twoside]{amsart}

\usepackage{amssymb,latexsym,amscd,amsmath, amsthm, bold-extra, enumitem, hyperref, dsfont, mathtools}
\usepackage{MnSymbol}
\usepackage[most]{tcolorbox}
\usepackage{pgf,tikz,pgfplots}
\pgfplotsset{compat=1.10}
\usepackage{mathrsfs}
\usetikzlibrary{arrows}

\usepackage{graphicx}
\usepackage[lofdepth,lotdepth]{subfig}

\numberwithin{equation}{section}
\usepackage{tikz}

\newtheorem{theorem}{Theorem}[section]
\newtheorem{lemma}[theorem]{Lemma}
\newtheorem{proposition}[theorem]{Proposition}
\newtheorem{corollary}[theorem]{Corollary}
\newtheorem{conjecture}[theorem]{Conjecture}

\newtheorem{question}[theorem]{Question}

\theoremstyle{definition}
\newtheorem{definition}[theorem]{Definition}

\newtheorem{remark}[theorem]{Remark}
\newtheorem{example}[theorem]{Example}

\numberwithin{equation}{section}

\usepackage{pgf,tikz,pgfplots}
\pgfplotsset{compat=1.15}
\usepackage{mathrsfs}
\usetikzlibrary{arrows}

\title[Rational normal curves and Hadamard products]{Rational normal curves and Hadamard products}
\thanks{Last updated: \today}

\author[E. Carlini]{Enrico Carlini}
\address[E. Carlini]{DISMA-Department of Mathematical Sciences \\
	Politecnico di Torino, Turin, Italy}
\email{enrico.carlini@polito.it}

\author[M. V. Catalisano]{Maria Virginia Catalisano}
\address[M. V. Catalisano]{Dipartimento di Ingegneria Meccanica, Energetica, Gestionale e dei
	Trasporti, Universit\`a degli studi di Genova, Genoa, Italy}
\email{catalisano@dime.unige.it}

\author[G. Favacchio]{Giuseppe Favacchio}
\address[G. Favacchio]{DISMA-Department of Mathematical Sciences \\
	Politecnico di Torino, Turin, Italy}
\email{giuseppe.favacchio@polito.it}

\author[E. Guardo]{Elena Guardo}
\address[E. Guardo]{Dipartimento di Matematica e Informatica\\
	Universit\`a degli studi di Catania\\
    Catania, Italy}
\email{guardo@dmi.unict.it}

\begin{document}

\begin{abstract}
Given $r>n$ general hyperplanes in $\mathbb P^n,$ a star configuration of points is the set of all the $n$-wise intersection of them.  We introduce {\it contact star configurations}, which are star configurations where all the hyperplanes are osculating to the same rational normal curve. In this paper we find a relation between this construction and Hadamard products of linear varieties. Moreover, we study the union of contact star configurations on a same conic in $\mathbb P^2$, we prove that the union of two contact star configurations has a special $h$-vector and, in some cases, this is a complete intersection.
\end{abstract}
\maketitle


	\section{Introduction}

	We say that the hyperplanes in a set $\mathcal{L}=\{\ell_1,\ldots,\ell_r\}\subseteq \mathbb P^n$, $r\ge n$,  {\it meet properly }if  $\ell_{i_1}\cap\cdots\cap \ell_{i_n}$ is a point for any choice of $n$ different indices and  $n+1$ hyperplanes are never concurrent. 
	We denote $\ell_{i_1}\cap\cdots\cap \ell_{i_n}$ by $P_{i_1,\ldots,i_n}$. 
	
	Let $\mathcal{L}=\{\ell_1,\ldots,\ell_r\}\subseteq \mathbb P^n$ be a set of $r\ge n$ hyperplanes meeting properly. The set of points
	\[\mathbb{S}(\mathcal{L})=\bigcup_{1 \le i_1<\ldots< i_n\le r}
	P_{i_1,\ldots, i_n}\subseteq\mathbb P^n.\]
	is called a {\it star configuration} of points in $\mathbb P^n$ defined by $\mathcal L.$
	
	These configurations of points, and their generalizations,  have been intensively studied for their algebraic and geometrical properties,
	see \cite{AS2012, BH2010comparing, CGVT2015, GHM2013, PS2015} 
	  for a partial list of papers that have contributed to our understanding them.

	Set  $S = \mathbb{C}[x_0,\ldots, x_n]= \mathbb{C}[\mathbb P^n]$, where $\mathbb{C}$ could be replaced by any algebraically closed field of characteristic zero.
	We recall that the Hilbert function of a set of points $X\subseteq \mathbb P^n$ is the numerical function $H_X: \mathbb Z_{\ge 0} \to \mathbb Z_{\ge 0}$ defined by
	\[H_X (t)= \dim S_t - \dim (I_X )_t,\]
	where $I_X$ is the ideal defining $X,$ and the $h$-vector of a set of points $X\subseteq \mathbb P^2$ is the first difference of the Hilbert function of $X$, that is
	\[
	h_X(t) = H_X(t)-H_X(t-1), 
	\]
	where we set $H_X(-1) = 0.$
	
	A star configuration $\mathbb{S}(\mathcal{L})$ defined by a set  of $r$ hyperplanes consists of $\binom{r}{n}$ points, and its $h$-vector is generic, see for instance \cite[Theorem 2.6]{GHM2013}, that means $h_{\mathbb{S}(\mathcal{L})}=\left(1,\ldots,\binom{n-1+i}{n-1}, \ldots, \binom{r-1}{n-1}\right).$ %
	Indeed, the ideal defining $\mathbb{S}(\mathcal{L})$ is minimally generated in degree $r-1$ by all the products of $r-1$ linear forms defining the hyperplanes in $\mathcal{L}$. 
	
	We now construct star configurations starting from a rational normal curve $\gamma$ of $\mathbb P^n$. We call them {\it contact star configurations on $\gamma$}, we will not mention $\gamma$ if it is clear from the context. 	
	
	\begin{definition}\label{d.contact star}
		Let $P_1, \ldots, P_r\subseteq \mathbb P^n$ be distinct points on a rational normal curve $\gamma$ of $\mathbb P^n$. Denoted by $\mathcal L=\{\ell_1, \ldots, \ell_r\}$ the set of osculating hyperplanes to $\gamma$ at $P_1, \ldots, P_r$ respectively. We say that $\mathbb{S}(\mathcal L)$ is a {\it  contact star configuration} on $\gamma$.
	\end{definition}
	Note that, since $\gamma$ is a rational normal curve, the hyperplanes in $\mathcal L$ always meet properly.
	
	The first motivation to introduce these configurations come from Hadamard products. We show in Section \ref{s. hadamard}  that the so called {\it Hadamard star configurations} are indeed contact star configurations, see Theorem \ref{t. hadamard construction}. This result will give an easy way to explicitly construct examples which only make use of rational points, see in Remark \ref{r. c i had}.
	
	The second motivation is related to their $h$-vector. A single contact star configuration has a generic $h$-vector, as any other star configuration. But the behavior of a union of two or more of them deserves further investigation. The homological invariants of a set of points union of star configurations have been studied for instance in \cite{AS2012, shin2012hilbert, shin2013some}. In the known cases, that require some restrictive assumptions, the $h$-vector of such a union is always general.

	We will mostly focus on  $\mathbb P^2$, therefore the contact star configurations are defined by taking lines tangent to an irreducible conic. The study of properties of families of lines tangent to a planar conic is  classical in algebraic geometry, see for instance the Cremona's book~\cite{cremona1885}.


We prove, in Section \ref{s. 2 contact},   that the union of two contact star configurations  in $\mathbb P^2$, defined by of $r$ and $s$ lines is a complete intersection of type $(r-1,s)$  if either $s=r-1$ or $s=r$, see Theorem \ref{t.main}.  We also show that, in these cases the curve of degree $s$ can be chosen irreducible.  

Moreover, in Section \ref{s.other unions}, we prove that the union of contact star configurations in $\mathbb P^2$ defined $r$ and $s$ lines has the same $h$-vector of two fat points of multiplicities $r-1$ and $s-1$, see Theorem \ref{t. h-vector 2 star }. 	
We believe that this correspondence with the $h$-vector of certain scheme of fat points also occurs for a union of three and four contact stars, see Conjecture \ref{conj. 3,4 star}. We prove it in some cases, see Theorem \ref{t. 3 fat 3 had}. 	

In Section \ref{s.applications}, we apply Theorem \ref{t.main} to the study of a recurring topic in classical projective geometry:  polygons circumscribed around an irreducible conic in $\mathbb P^2$,  see Proposition \ref{p.brianchonlike}, Corollary \ref{c.conic8points} and Proposition \ref{p.3conics}.

Section \ref{s. future} contains concluding remarks and conjectures for further work.

We will make use of standard tools from the linkage theory, see \cite{MN20} for an overview in the topic and \cite{FGM2018, FM2019, GHM2013,  kreuzer2019, MN3}  for a partial list of papers which use liaison to study zero-dimensional projective and mutiprojective schemes. 
A well know result, see \cite[Corollary 5.2.19]{MiglioreBook}, relates the $h$-vectors of two  arithmetically Cohen-Macaulay schemes in $\mathbb P^n$ with the same codimension, that are linked by an  arithmetically Gorenstein. In particular, if $X,Y$ are two disjoint set of reduced points in $\mathbb P^2$ and $X\cup Y$ is a complete intersection of type $(a,b)$, then the following formula connects the $h$-vectors of $X$, $Y$ and $X\cup Y$ 
\begin{equation}\label{eq.Corollary 5.2.19}
h_{X\cup Y}(t)=h_X(t)+h_Y(a+b-2-t), \ \ \text{ for any integer } t.
\end{equation}
Since the $h$-vector of a complete intersection is well known, from the formula above, the knowledge of the $h$-vector of $X$ allows to compute the one of $Y$.

\noindent{\bf Acknowledgments.} Favacchio and Guardo  have been supported by Università degli Studi di Catania, piano della ricerca PIACERI 2020/22 linea intervento 2.
 All the authors have been supported by the National Group for Algebraic and Geometrical Structures and their Applications (GNSAGA-INdAM).  
This work was partially supported by MIUR grant Dipartimenti di Eccellenza 2018-2022 (E11G18000350001).
 Our results were inspired by calculations with CoCoA~\cite{CoCoA}.

\section{Hadamard products}\label{s. hadamard}

In this section we show that Hadamard star configurations are contact star configurations. 
Hadamard products of linear spaces have been recently subject of study for many interesting properties, see for instance  \cite{BCK2016, CCGVT2019, CGVT2015}.
We briefly recall some general fact about Hadamard products of linear spaces. 
Let $P = [a_0:\cdots:a_n]$ and $Q = [b_0:\cdots:b_n]$ be two
points in $\mathbb{P}^n$.  If for some $i$ we have both $a_i\neq 0$ and $b_i\neq 0$, then we say that the {\it Hadamard product} of $P$ and $Q$, 
denoted $P \star Q$, is defined and we set
\[P \star Q = [a_0b_0: \cdots : a_nb_n]\in \mathbb{P}^n.\]

Given two varieties $X$ and $Y$ in $\mathbb{P}^n$, with respect to the Zariski topology,  the {\it Hadamard product} of $X$ and $Y$, denoted $X \star Y$, is given by 
\[X \star Y = \overline{\{ P \star Q\ |\ P \in X, Q \in Y, \ \ \mbox{and} \ \
	P \star Q\ \mbox{is defined}\}}\subseteq \mathbb{P}^n.\]	

In particular, for a variety $X$ in $\mathbb P^n$ and a positive integer $r\ge 2$,  the $r$-th Hadamard power of $X$ is 
$$X^{\star r}=X^{\star (r-1)}\star X,$$ where we define $X^{\star 1}=X$. 

When we compute the Hadamard product of $X$ and $Y$ it often crucial to ensure some condition of generality on $X$ and $Y$, this is encoded by not having too much zero coordinates in their points.  For this purpose, we let $\Delta_i$ be the set of points of $\mathbb{P}^n$ which have at most $i+1$ non-zero coordinates.

A slightly different definition of Hadamard product is given in Definition 2.15 \cite{CCGVT2019}. If ${X}$ is a finite set of points in $\mathbb{P}^n$, then
the $r$-th {\it square-free Hadamard product} of $\mathbb{X}$
is 	\[X^{^{\underline{\star} r}} = \{P_1 \star \cdots \star P_r ~|~
P_1,\ldots, P_r \in X ~~\mbox{distinct \ points}\}.\]
From \cite[Theorem 4.7]{BCK2016}  it is known that $\mathbb{X}^{^{\underline{\star} n}}$ is a star configuration of $\binom{m}{n}$ points of $\mathbb{P}^n$, where $X\subseteq \mathbb P^n$ is a set of $m>n$ points on a line $\ell$ such that  $\ell \cap \Delta_{n-2}
= \emptyset$.

Let $V$ be a linear space in $\mathbb P^n$, for a positive integer $r$ we consider the subscheme  $$V^{\circ r}=\{ P^{\star r}\ |\  P\in V \}\subseteq \mathbb P^n,$$ called the {\it $r$-th coordinate-wise power} of $V$. Properties of these schemes have been studied in \cite{coordinate-wise}. In the following theorem, we investigate the case where $V$ is a line. 

\begin{theorem}\label{t. hadamard construction}
	Let $\ell$ be a line in $\mathbb P^n$ such that $\ell \cap \Delta_{n-2}
	= \emptyset$. Then
	\begin{itemize}
		\item[\rm {(i)}] $\ell^{\circ n}= \{ P^{\star n} | P \in \ell\}$ is a rational normal curve;
		\item[\rm {(ii)}] let $P \in \ell$, then the linear subspace of dimension $d$ osculating to $\ell^{\circ n}$ at $P^{\star n}$ (i.e., its intersection with $\ell^{\circ n}$ is supported only on $P$) is $P^{\star (n-d)}\star \ell^{\star d}$. In particular,  the osculating hyperplane to $\ell^{\circ n}$ at $P^{\star n}$ is $P\star \ell^{\star (n-1)}$;
		\item[\rm {(iii)}] for each set of $n$ distinct points on $\ell$, $P_1,\ldots ,P_{n}\in \ell,$ we have $P_1\star P_2 \star \cdots \star P_n=P_1\star \ell^{\star (n-1)}\cap \cdots \cap P_n\star \ell^{\star (n-1)}.$
	\end{itemize}
\end{theorem}
\begin{proof} The degree of $\ell^{\circ n}$ is $n$ from Corollary 2.8 in \cite{coordinate-wise}.  Take a parametrization of the line $\ell$, say $P_{ab}=[L_0(a,b): L_1(a,b): \ldots: L_n(a,b)]\in \ell$ where the $L_i(a,b)$ are linear forms in the variables $a,b.$ Note that, since  $\ell \cap \Delta_{n-2}$ is empty, the forms $L_i(a,b)$ are pairwise not proportional.
	\begin{itemize}
		\item [(i)]  The curve $\ell^{\circ n}$ is  parametrized by $P_{ab}^{\star n}=[L_0(a,b)^n: L_1(a,b)^n: \ldots: L_n(a,b)^n]\in \ell^{\circ n}$, where the components of $P_{ab}^{\star n}$ are a basis for the forms of degree $n$ in $a,b$ since the $L_i(a,b)$ are pairwise not proportional. Hence, $\ell^{\circ n}$ is a rational normal curve of $\mathbb P^n$.
		\item [(ii)] We will prove item (ii) by induction on $d$. Let $d=1$.  Now let $P+t Q $ be a point of $ \ell$, ($t \in \mathbb C$), thus the tangent line to  $\ell^{\circ n}$ at $P^{\star n}$ is
		\[
		\lim_{t \to 0}\ \langle\ P^{\star n}, \ (P+tQ)^{\star n}\ \rangle =
		\lim_{t \to 0}\ \langle\ P^{\star n}, \ P^{\star n}+ ntP^{\star (n-1)}\star Q+ \dots +t^nQ^{\star n}\ \rangle \]
		\[=  \langle\ P^{\star n}, \ P^{\star (n-1)}\star Q\ \rangle = P^{\star (n-1)}\star \ell.
		\]
		Assume $d>1$. By the induction hypothesis, the linear space of dimension $d-1$  osculating to  $\ell^{\circ n}$ at $P^{\star n}$ is  $P^{\star (n-d+1)}\star \ell^{\star (d-1)}$. Let $Q\neq P$ be a point on $\ell$. We have
		\[P^{\star (n-d+1)}\star \ell^{\star (d-1)}
		\]
		\[=\{P^{\star (n-d+1)}\star (a_1P+b_1Q)\star \ldots \star (a_{d-1}P+b_{d-1}Q)   \ | \ a_i , b_i \in \mathbb C\}\]
		\[= \langle\ P^{\star n}, \  P^{\star (n-1)}  \star Q, \  P^{\star (n-2)}  \star Q^{\star 2}, \ldots, P^{\star (n-d+1)}  \star Q^{\star (d-1)} \rangle \ .  \]
		Now let again  $P+t Q $ be a point of $ \ell$, ($t \in \mathbb C$). The linear space of dimension $d$ osculating to $\ell^{\circ n}$ at $P^{\star n}$ can be obtained by computing the following limit
		\[
		\lim \limits_{t\to 0} \langle\ P^{\star (n-d+1)} \star \ell^{\star (d-1)}, \ (P+tQ)^{\star n}
		\ \rangle ,\]
		and this limit, by an easy computation and the equality above, becomes
		\[
		\begin{array}{rll}
		& \lim \limits_{t\to 0} \langle\
		P^{\star n}, \  P^{\star (n-1)}  \star Q, \  P^{\star (n-2)}  \star Q^{\star 2}, \ldots
		,P^{\star (n-d+1)}  \star Q^{\star (d-1)}, \ (P+tQ)^{\star n}
		\ \rangle \\
		= & \lim \limits_{t\to 0} \langle\
		P^{\star n}, \  P^{\star (n-1)}  \star Q,  \ldots
		,P^{\star (n-d+1)}  \star Q^{\star (d-1)}, \  P^{\star n}+ ntP^{\star (n-1)}\star Q+ \dots +t^nQ^{\star n} \ \rangle \\
		= & \lim \limits_{t\to 0}   \langle\
		P^{\star n}, \  P^{\star (n-1)}  \star Q,  \ldots
		,P^{\star (n-d+1)}  \star Q^{\star (d-1)}, \  {n \choose d} t^{d}P^{\star (n-d)}\star Q^{\star d}+ \dots +t^nQ^{\star n} \ \rangle \\
		= &  \langle\
		P^{\star n}, \  P^{\star (n-1)}  \star Q,  \ldots
		,P^{\star (n-d+1)}  \star Q^{\star (d-1)}, \ P^{\star (n-d)}\star Q^{\star d} \ \rangle \\
		= &P^{\star (n-d)}\star \ell^{\star d}.
		\end{array}
		\]
		\item [(iii)] It follows from (ii) and a well known property of rational normal curves in $\mathbb P^n.$
\end{itemize}\end{proof}
%
In the next remark we give more details for $n=2$.
\begin{remark} Consider a line $\ell$ in $\mathbb P^2$ and the respective conic $\ell^{\circ 2}$, let $\mathbb C[x,y,z]$ be the coordinate ring of $\mathbb P^2$. We have the following facts. 
	
	\begin{itemize}
		\item[(i)] Say $\ell$ defined by the equation $\alpha x+ \beta y-z=0$, where $\alpha, \beta\neq 0$. Then, from Theorem \ref{t. hadamard construction}(i) we have that $\ell^{\circ 2}$ is a conic. Precisely, one can check that \[\ell^{\circ 2}:\ \ \ (\alpha^2 x+ \beta^2 y-z)^2-4\alpha^2\beta^2 xy=0.\]
		
		\item[(ii)] From Theorem \ref{t. hadamard construction}(ii), for each $P\in \ell$ the line $P\star \ell$ is tangent to $\ell^{\circ 2}$ at $P\star P$.
		
		\item[(iii)] From Theorem \ref{t. hadamard construction}(iii),  for any $P, Q\in \ell$, $P\neq Q$,  the two tangent lines to $\ell^{\circ 2}$ through $P \star Q$ are $P \star \ell $ and $Q \star \ell$. Note that this allows us to find an explicit Hadamard decomposition of any point  in the plane $\mathbb P^2=\ell \star \ell$. In fact,  let  $A\in \mathbb P^2$,  let $a$ and $b$ be the tangent lines to the conic $\ell^{\circ 2}$ through $A$, and let $P\star P= a \cap \ell^{\circ 2}$, $Q \star Q= b \cap \ell^{\circ 2}$, then $A=P \star Q $.
				
		\item[(iv)] For any $P, Q\in \ell$, $P\neq Q$, the  line through $P\star P$ and $Q\star Q$  is the polar line of the point $P\star Q$ with respect to $\ell^{\circ 2}$.
	
		\item[(v)] The condition $\ell \cap \Delta_{0} = \emptyset$ ensures that $\ell$ meets the lines $x=0$, $y=0$ and $z=0$ in three distinct points, say $P_x, P_y$ and $P_z$, respectively. Note that, from the definition of Hadamard product, $P_x*\ell$ is the line $x=0$ and analogously $P_y*\ell$ is the line $y=0$ and $P_z*\ell$ is $z=0$. Then, the conic $\ell^{\circ 2}$ is tangent to the coordinate axes in $P_x*P_x$, $P_y*P_y$ and $P_z*P_z$.
	\end{itemize}
\end{remark}
\section{Complete intersections union of two contact star configurations in $\mathbb P^2$ }\label{s. 2 contact}
In this section $\gamma$ is an irreducible conic in $\mathbb P^2$, and we set $S=\mathbb C[x,y,z]=\mathbb C[\mathbb P^2]$. 
The main result of this section is the following Theorem. We postpone its proof until page \pageref{proof.main}, after the development of some special case.

\begin{theorem}\label{t.main} Let $X=\mathbb{S}(\mathcal{L})$ and $Y=\mathbb{S}(\mathcal{M})$  be two contact star configurations in $\mathbb P^2$ on the same conic, where $\mathcal{L}=\{\ell_1,\ldots,\ell_r\}$ and $\mathcal{M}=\{m_1,\ldots,m_s\}$ are two sets of distinct lines. 
	Then

\begin{itemize}
	\item[\rm{(a)}] if $s=r-1$, then general form in $\left(I_{X\cup Y}\right)_{r-1}$ is irreducible;
	\item[\rm{(b)}] if $s=r-1$, then $X\cup Y$ is a complete intersection of type $(r-1,r-1)$;
	\item[\rm{(c)}] if $s=r$, then the general form in $\left(I_{X\cup Y}\right)_{r}$ is irreducible;
	\item[\rm{(d)}] if $s=r$, then  $\left(I_{X\cup Z}\right)_{r-1}=\left(I_{X\cup Y}\right)_{r-1}$ where $Z$ denotes a set of $r-1$ collinear points in $Y$;
	\item[\rm{(e)}] if $s=r$, then $X\cup Y$ is a complete intersection of type $(r-1,r)$.
\end{itemize}
\end{theorem}

\begin{remark}
	Theorem \ref{t.main} (e) in particular claims that the twelve points union of two contact star configurations on the same conic, $X=\mathbb S(\ell_1, \ell_2, \ell_3,\ell_4)$ and $Y=\mathbb S(m_1, m_2, m_3,m_4)$ lie on a cubic. This case is pictured in Figure \ref{fig.cubic12points}.
	
	\begin{figure}[h]
		\centering
		\includegraphics[scale=0.3]{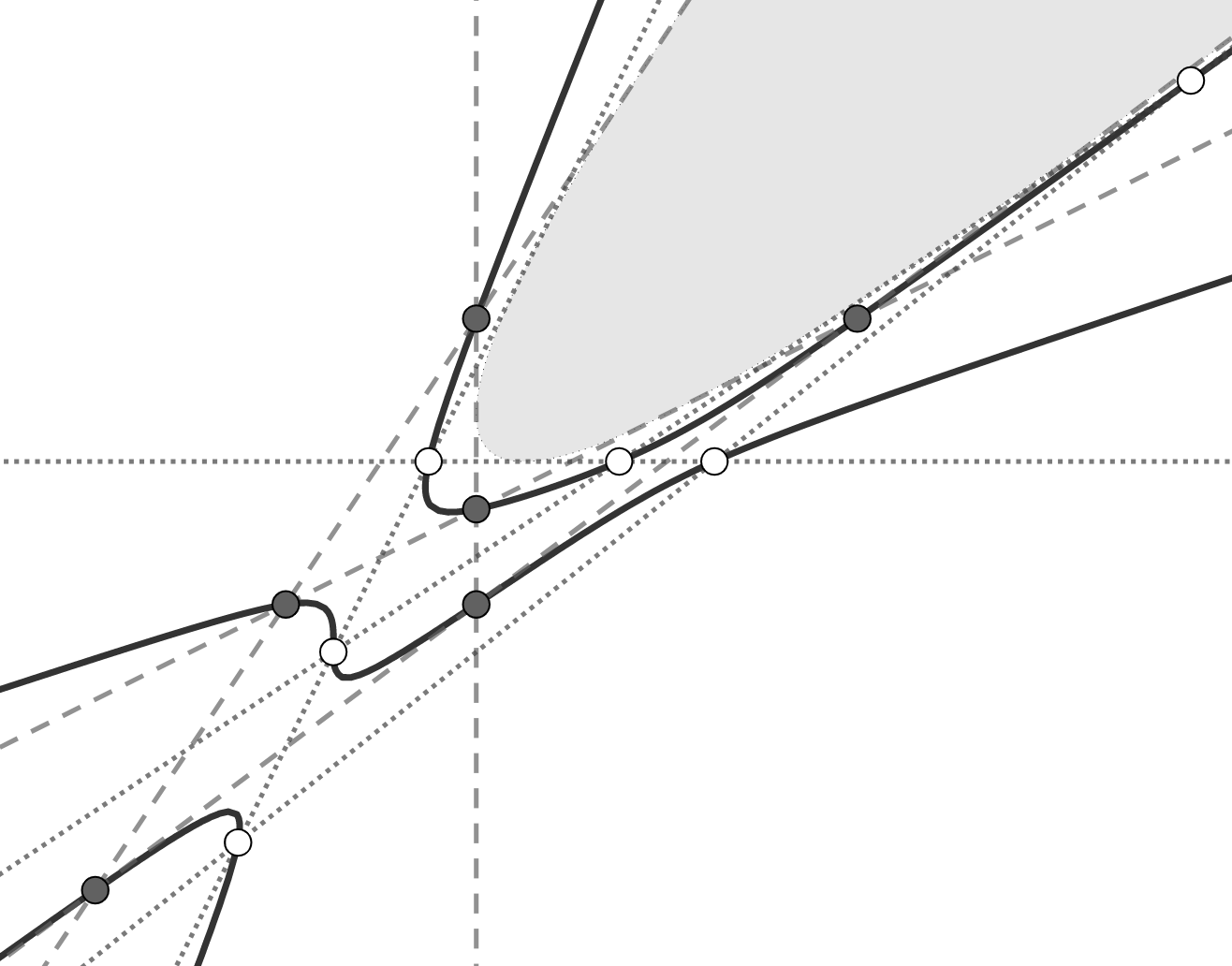}
		\caption{ A cubic through twelve points union of two contact star configurations. }\label{fig.cubic12points}
	\end{figure}
\end{remark} 

\begin{remark}\label{r. c i had}
	Combining Theorem \ref{t.main} and Theorem \ref{t. hadamard construction} we are able to explicitly give the coordinates of a (not trivial) complete intersection of rational points in $\mathbb P^2$ of type $(a,b)$, where either $b=a$ or $b=a-1$.  First, we need to fix a rational line $\ell$ in $\mathbb P^2$. Then, we take $a+b+1$ distinct rational points on $\ell$ divided in two sets $X$ and $Y$ containing $a+1$ and $b$ points, respectively. Then  $\mathbb{X}^{^{\underline{\star} 2}}\cup \mathbb{Y}^{^{\underline{\star} 2}}$ is the complete intersection we are looking for.  
	For instance, let $\ell$ be defined by the linear form $x+y-z$ and $a=b=3$. We pick on $\ell$ the following points 
	$$X=\{[1:1:2],[1:2:3],[1:3:4],[1:4:5] \} \ \ 
	\text{and} 
	\  \ Y=\{[1:-1:0],[1:-2:-1],[1:-3:-2]\}.$$ 
	Then, the set of 9 points 
	$$\begin{array}{rl}
	\mathbb{X}^{^{\underline{\star} 2}}\cup \mathbb{Y}^{^{\underline{\star} 2}}=&\{ [1:2:6], [1:3:8], [1:4:10], [1:6:12], [1:8:15], [1:12:20] \}\ \cup\\ 
	&\{[1:2:0], [1:3:0], [1:6:2]\}
	\end{array} $$
	from Theorem \ref{t.main} (b), is a complete intersection of type $(3,3)$.
\end{remark}

\begin{remark}\label{r. cases s =2}
Note that if $r=s=2$, then $X\cup Y$ consists of two points and the statements (c), (d), (e) in Theorem \ref{t.main} are trivially true.  The statements (a), (b) in the case $r=3, s=2$ are also trivial. Indeed, $X\cup Y$ consists of four points in linear general position. The first interesting case occurs when $r=s=3$ (see Figure \ref{fig.3-3}).
\end{remark}

In the following lemma we prove Theorem \ref{t.main} in case $r=s=3$.
\begin{lemma}\label{l. case 3-3}
Let $X=\mathbb{S}(\ell_1,\ell_2,\ell_3)$ and $Y=\mathbb{S}(m_1,m_2,m_3)$ be contact star configurations on a same conic.  Then
	\begin{itemize}
		\item[\rm{(i)}]   the general cubic through $X\cup Y$ is irreducible;
		\item[\rm{(ii)}]  a conic containing 5 points of $X\cup Y$ contains  $X\cup Y$;  
		\item[\rm{(iii)}]  $X\cup Y$ is a complete intersection of a conic and a cubic.
	\end{itemize}

\end{lemma}
\begin{proof}
	Since the linear system of the cubics through $X\cup Y$ is not composite with a pencil and it doesn't have a common component, then by Bertini's Theorem, see for instance \cite[Section 5]{kleiman1998} and \cite[Corollary 10.9]{hartshorneAG},  the generic cubic of the system is irreducible.
   In order to complete the proof, let $C_3, C_3'$ be two cubics union of lines through $X\cup Y$, (see Figure \ref{fig.3-3-brian}). 
	The intersection $C_3\cap C_3'$ consists of 9 points and,  by Brianchon's Theorem, see \cite[pp. 146-147]{casey1888sequel}, the three of them  not lying in $X\cup Y$  are on a line (the white circles in Fig. \ref{fig.3-3-brian}).  Thus, by liaison (use formula \ref{eq.Corollary 5.2.19}) the set $X\cup Y$ is contained in a conic.
\end{proof}

\begin{figure}[h]
	\centering
	\subfloat[Case $r=s=3.$]{
	\includegraphics[scale=0.47]{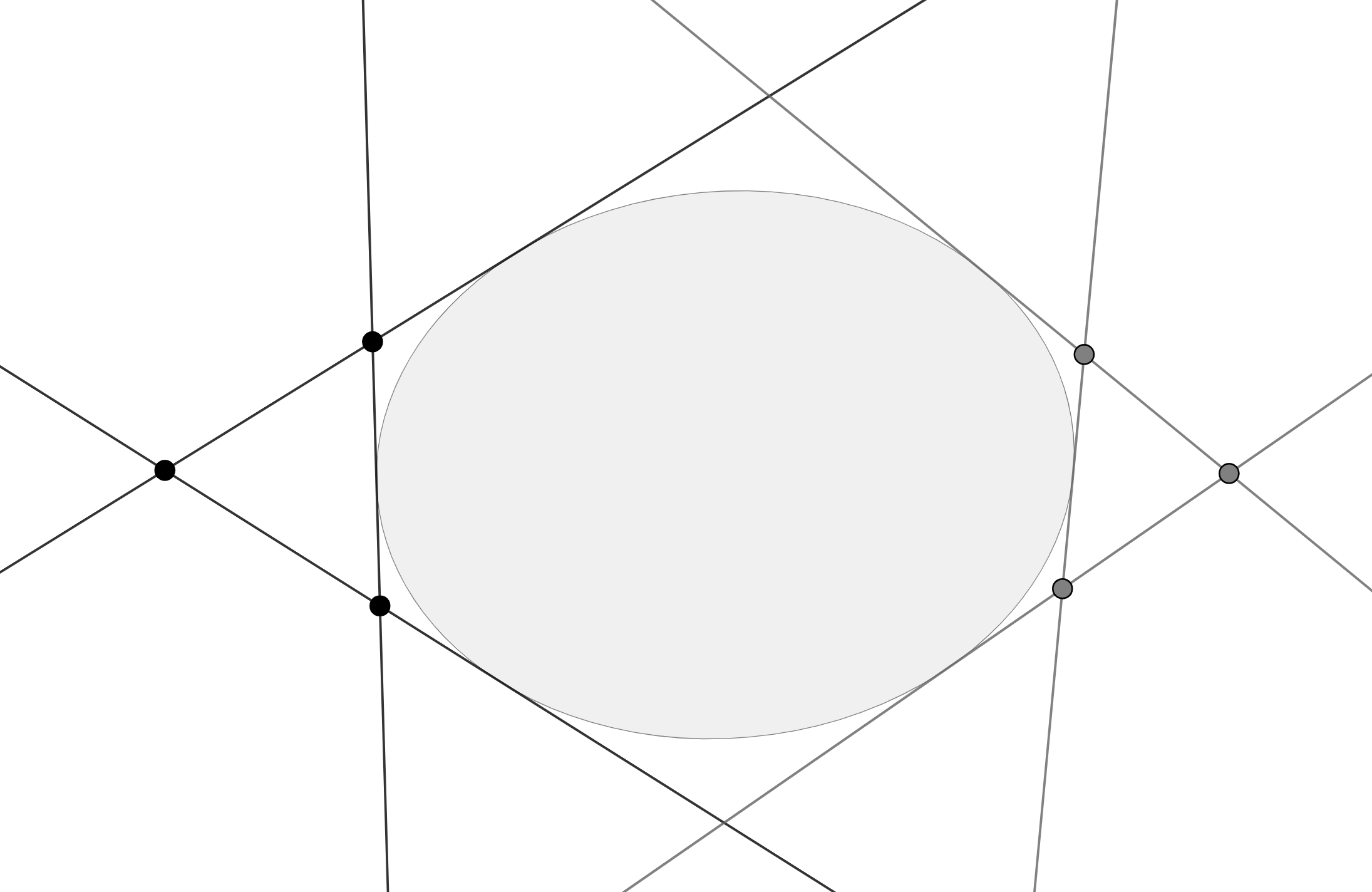}
\label{fig.3-3}}
	\qquad
	\subfloat[The nine points intersection of the two cubics.]{
		\includegraphics[scale=0.47]{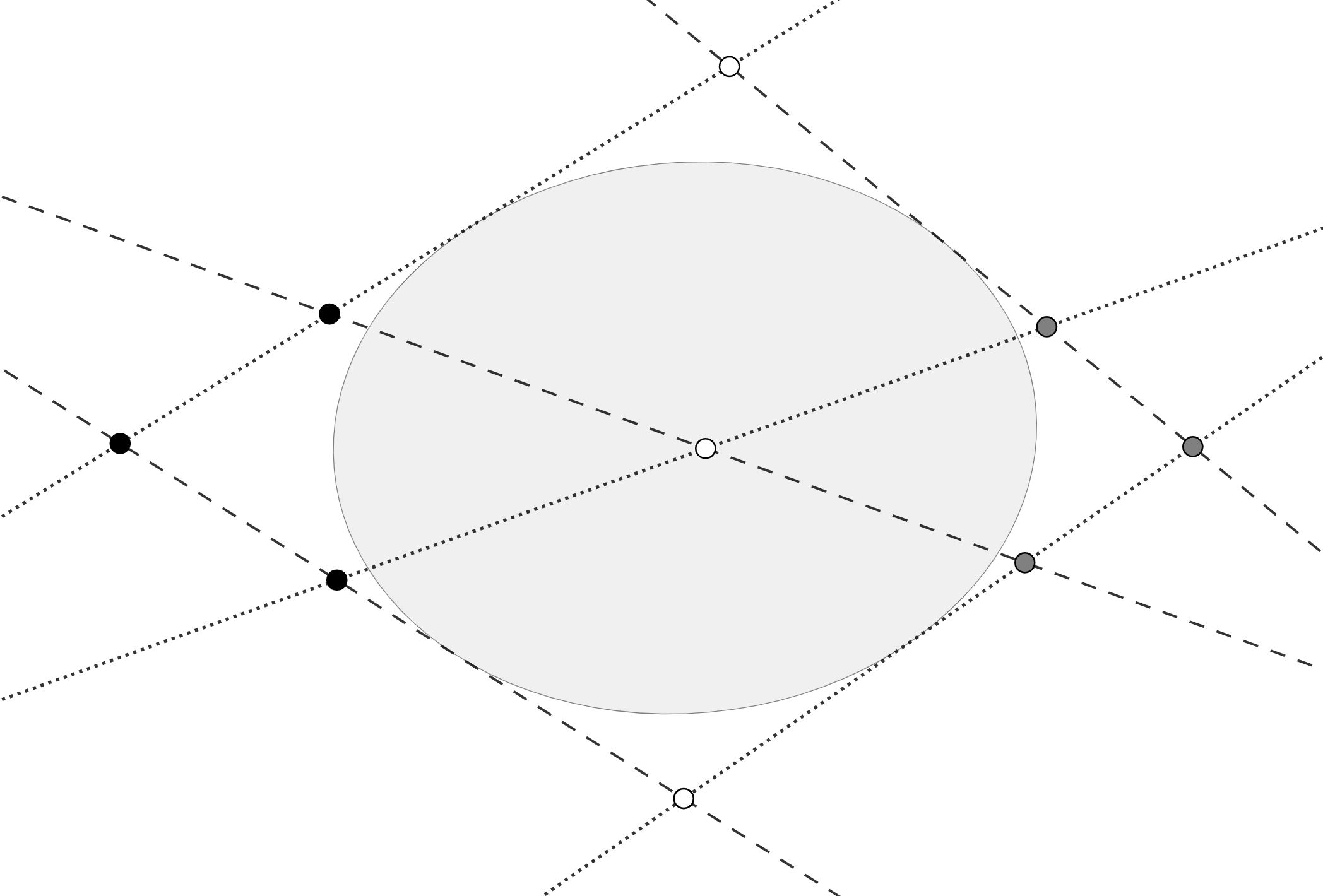}
		\label{fig.3-3-brian}}
	\caption{ }
	\label{fig:case33}
\end{figure}

Now we are ready to prove Theorem \ref{t.main}. We recall its statement.

\noindent{\bf Theorem \ref{t.main}}{\it\ Let $X=\mathbb{S}(\mathcal{L})$ and $Y=\mathbb{S}(\mathcal{M})$ be two contact star configurations in $\mathbb P^2$ on the same conic, where $\mathcal{L}=\{\ell_1,\ldots,\ell_r\}$ and $\mathcal{M}=\{m_1,\ldots,m_s\}$ are two sets of distinct lines. 
	Then	
	\begin{itemize}
		\item[\rm{(a)}] if $s=r-1$, then general form in $\left(I_{X\cup Y}\right)_{r-1}$ is irreducible;
		\item[\rm{(b)}] if $s=r-1$, then $X\cup Y$ is a complete intersection of type $(r-1,r-1)$;
		\item[\rm{(c)}] if $s=r$, then the general form in $\left(I_{X\cup Y}\right)_{r}$ is irreducible;
		\item[\rm{(d)}] if $s=r$, let $Z= Y\cap m_r$, then  $\left(I_{X\cup Z}\right)_{r-1}=\left(I_{X\cup Y}\right)_{r-1}$;
		\item[\rm{(e)}] if $s=r$, then $X\cup Y$ is a complete intersection of type $(r-1,r)$.
	\end{itemize}
	}

\begin{proof}\label{proof.main} We proceed by induction on $r$.  The cases $r\le 3$ follow by Remark \ref{r. cases s =2} and Lemma \ref{l. case 3-3}.  
So assume $r>3$. Set $X^{(i)}=\mathbb{S}(\mathcal{L}\setminus\{\ell_i\})$, thus $X^{(i)}$ is a star configuration of points defined by $r-1$ lines.
\begin{itemize}
	\item[(a)] Consider the sets $X^{(i)}\cup  Y$, for $i=1,\ldots, r-1$. By   (e) and by induction, $X^{(i)}\cup  Y$ is complete intersection of type $(r-2,r-1)$, thus there exists a curve of degree $r-2$ through  $X^{(i)}\cup  Y$, say $C_{r-2}^{(i)}$. 
	
	Note that $C_{r-2}^{(i)}$ does not have $\ell_r$ as a component. 
	In fact,	
	if $\ell_r$ is a component of $C_{r-2}^{(i)}$, then, by removing $\ell_r$, since it does not contain points of $Y$, we get a curve of degree $r-3$ through $Y$. A contradiction, since $I_Y$ starts in degree $r-2.$
	
	Since $\ell_r$ is not a component for $C_{r-2}^{(i)}$, then $C_{r-2}^{(i)}$ meets $\ell_r$ in exactly $r-2$ points that are	 $(X\cap \ell_r)\setminus( \ell_{i}\cap \ell_{r})$.
	
	From here it easily follows that the linear system of curves
	\begin{equation}\label{ls}
	\left\langle \ell_i\cup C_{r-2}^{(i)}\ |\ i=1,\ldots, r-1\right\rangle
	\end{equation} 
	does not have any fixed component and it is not composite with a pencil. Thus by Bertini's Theorem the general curve in \refeq{ls} is irreducible.

%
%
%

	\item[(b)] By (a), since the linear system  \eqref{ls} has dimension at least 2,  there exist two irreducible curves of degree $r-1$ thorough $X\cup Y$. Since  $|X\cup Y|= \binom{r}{2}+\binom{r-1}{2}=(r-1)^2$, we are done.

	\item[(c)] From (a) the generic curve of degree $r-1$ through $X^{(i)}\cup  Y$  is irreducible, say $C_{r-1}^{(i)}$, for each $i=1,\ldots, r.$ Then, the linear system  $\left\langle \ell_i \cup C_{r-1}^{(i)}\ |\ i=1,\ldots, r\right\rangle$ does not have any fixed component and it is not composite with a pencil. Again for Bertini's Theorem we are done.
	\item[(d)]  Let  $P_{ir}=\ell_i\cap \ell_r$ and  $Y^{(i)}=\mathbb{S}(\mathcal{M}\setminus\{m_i\})$, $i=1,\ldots, r-1$. From (b) the set $X\cup Y^{(i)}$ is a complete intersection of two curves of degree $r-1,$  hence 
	$$\dim_k \left(I_{X\cup Y^{(i)}}\right)_{r-1}=2$$ and
	 its $h$-vector is 
	\[
	h_{X\cup Y^{(i)}}=(1,2,3, \ldots, r-2, r-1, r-2, \ldots, 2,1).
	\]

	Moreover, $Y^{(i)}\setminus Z=\mathbb S(\mathcal M \setminus \{m_i,m_r\})$ is a star configuration defined by $r-2$ lines and then its $h$-vector is
	\[
	h_{Y^{(i)}\setminus Z}=(1,2,3, \ldots, r-3).
	\]
	Thus by liaison, see relation \eqref{eq.Corollary 5.2.19},  we have
	\[ 
	h_{X\cup (Z\setminus \{P_{ir}\})}=(1,2,3, \ldots, r-2, r-1, r-2),
	\]
	then
	$$\dim_k \left(I_{X\cup (Z\setminus \{P_{ir}\})}\right)_{r-1}=2.$$ 
	
		Since $X\cup (Z\setminus \{P_{ir}\})\subseteq X\cup Y^{(i)}$ 
		we have
	\begin{equation}\label{eq. ideal equal r-1}
	\left(I_{X\cup (Z\setminus \{P_{ir}\})}\right)_{r-1}=\left(I_{X\cup Y^{(i)}}\right)_{r-1}.
	\end{equation}
	
Note that in order to prove that $\dim \left(I_{X\cup (Z\setminus \{P_{ir}\})}\right)_{r-1}=2$, we can use Lemma 2.2 in~\cite{CCGbipolynomial} instead of liaison.	
	
%
%
	Let  $F\in (I_{X\cup Z})_{r-1}$,  then, by \eqref{eq. ideal equal r-1},  $F\in (I_{X\cup Y^{(i)}})_{r-1}$ for each $i=1,\ldots,r-1$.
	So, $F\in (I_{X\cup Y})_{r-1}$. It follows that $(I_{X\cup Y})_{r-1}=(I_{X\cup Z})_{r-1}$.

	\item[(e)]  Let $F\in \left( I_{X\cup Y} \right)_{r-1}$ as in the proof in item (d).
	By item (c) there exists an irreducible form $G\in \left( I_{X\cup Y}  \right)_{r}.$
	Since $$|F\cdot G|=r(r-1) = |X\cup Y|,$$ then $X\cup Y$ is a complete intersection of the curves defined by $F$ and $G.$
\end{itemize} 
\end{proof}
A natural question related to Theorem \ref{t.main} arises about the irreducibility of the curve of degree $r-1$ in the case $r=s$. It cannot be guaranteed. Indeed, in the next example,  we produce a set of 20 points (case $r=s=5$,) complete intersection of a quartic and a quintic, where the curve of degree 4 satisfying item (e) of Theorem \ref{t.main} is  union of two conics. 	

\begin{example}\label{e. not irred r-1}
	 Let $\mathcal{L}=\{\ell_1,\ldots,\ell_5\}$ be a set of five lines tangent to an irreducible conic $\gamma$ (the gray parabola in Figure \ref{fig:not-irred}). 
	The contact star configuration $X=\mathbb{S}(\mathcal{L})$ consists of ten points.
	
	We  split the ten points in two sets of five points, each contained in an irreducible conic. Then, we consider the quartic union of these two conics (the one dashed and the other dotted in Figure~\ref{fig:not-irred}). 
	
	\begin{figure}[h]
			\centering
			\includegraphics[width=.4\linewidth]{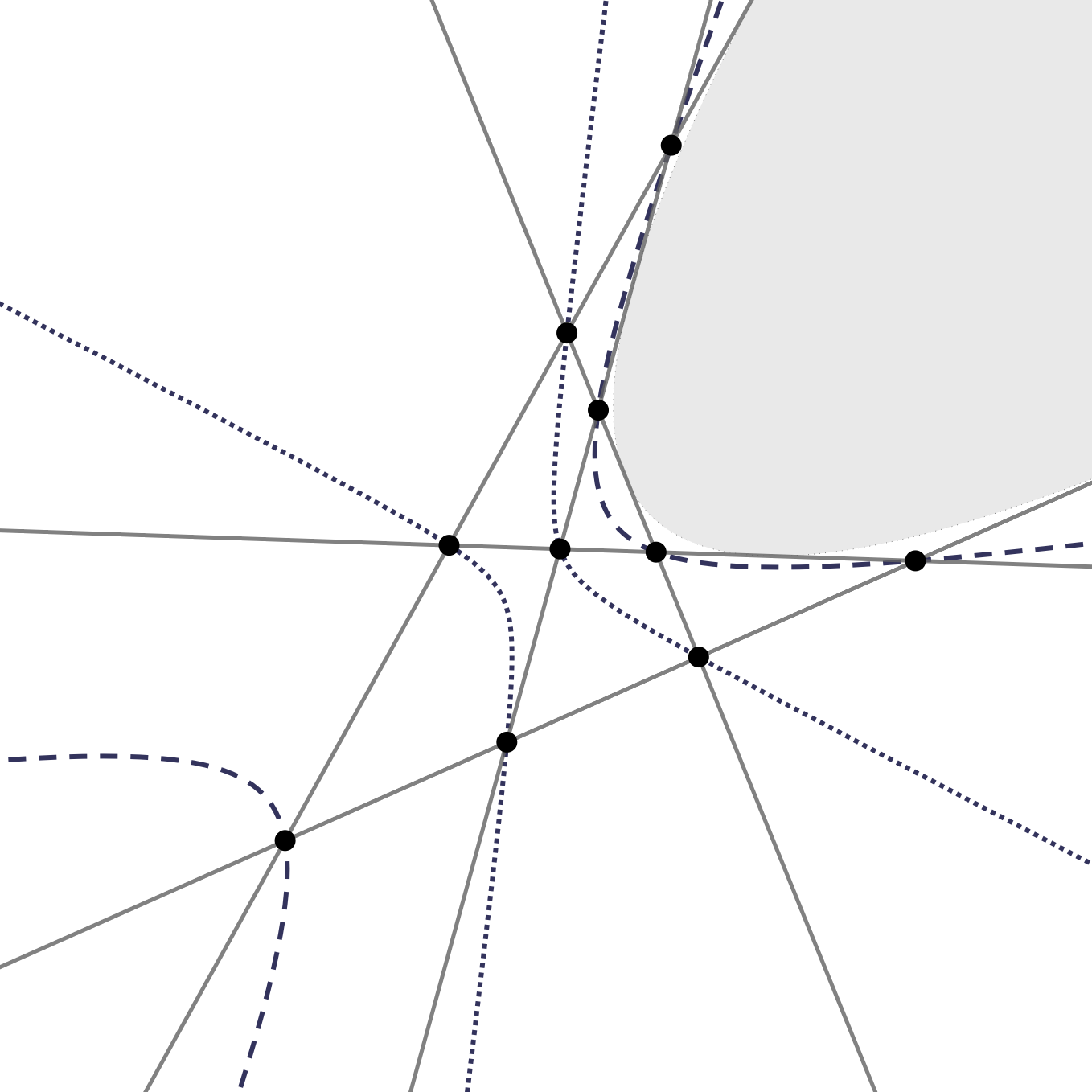}
			\caption{Example \ref{e. not irred r-1}: The star configuration $X$ and the not irreducible quartic through $X$.}
			\label{fig:not-irred}
		\end{figure}
	
	Now, we take a new line $m$ tangent to $\gamma$, see Figure \ref{fig:not-irred-0}. Through each of the four points intersection of $m$ with the quartic, there is an extra tangent line  to $\gamma$. Call these tangent liness $m_1, m_2,m_3,m_4$, see Figure \ref{fig:not-irred-1}. The set of points $Y=\mathbb{S}(\{m_1, \ldots, m_4, m\})$ is a contact star configuration and, from Theorem \ref{t.main} (d), $X\cup Y$ is contained in the quartic. Of course this quartic, by construction, is not irreducible.

\begin{figure}[h]
	\centering
	\subfloat[The line $m$ intersecting the quartic in four points.]{
		\includegraphics[width=0.41\textwidth]{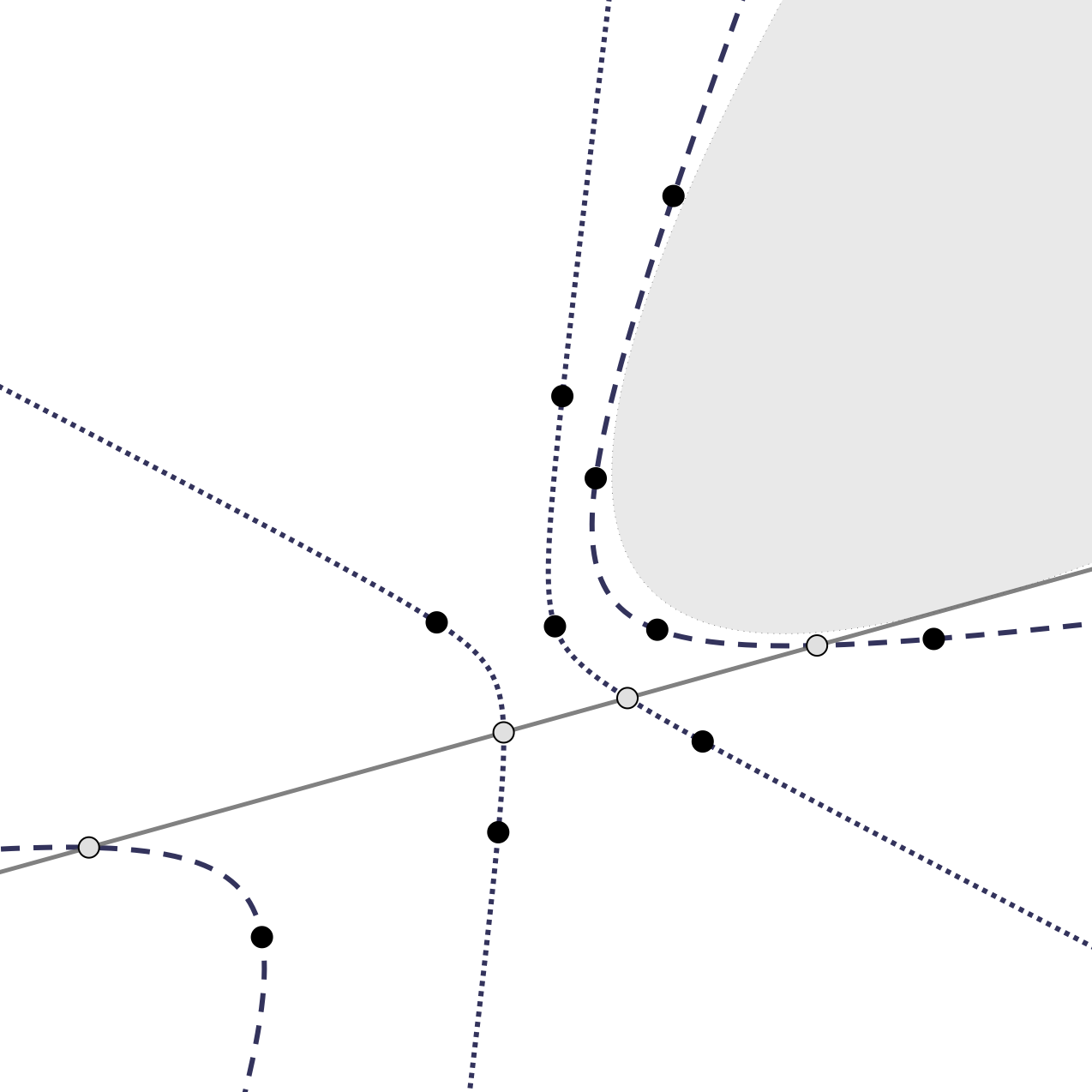}
		\label{fig:not-irred-0}}
	\qquad
	\subfloat[The contact star configuration $Y$.]{
		\includegraphics[width=0.41\textwidth]{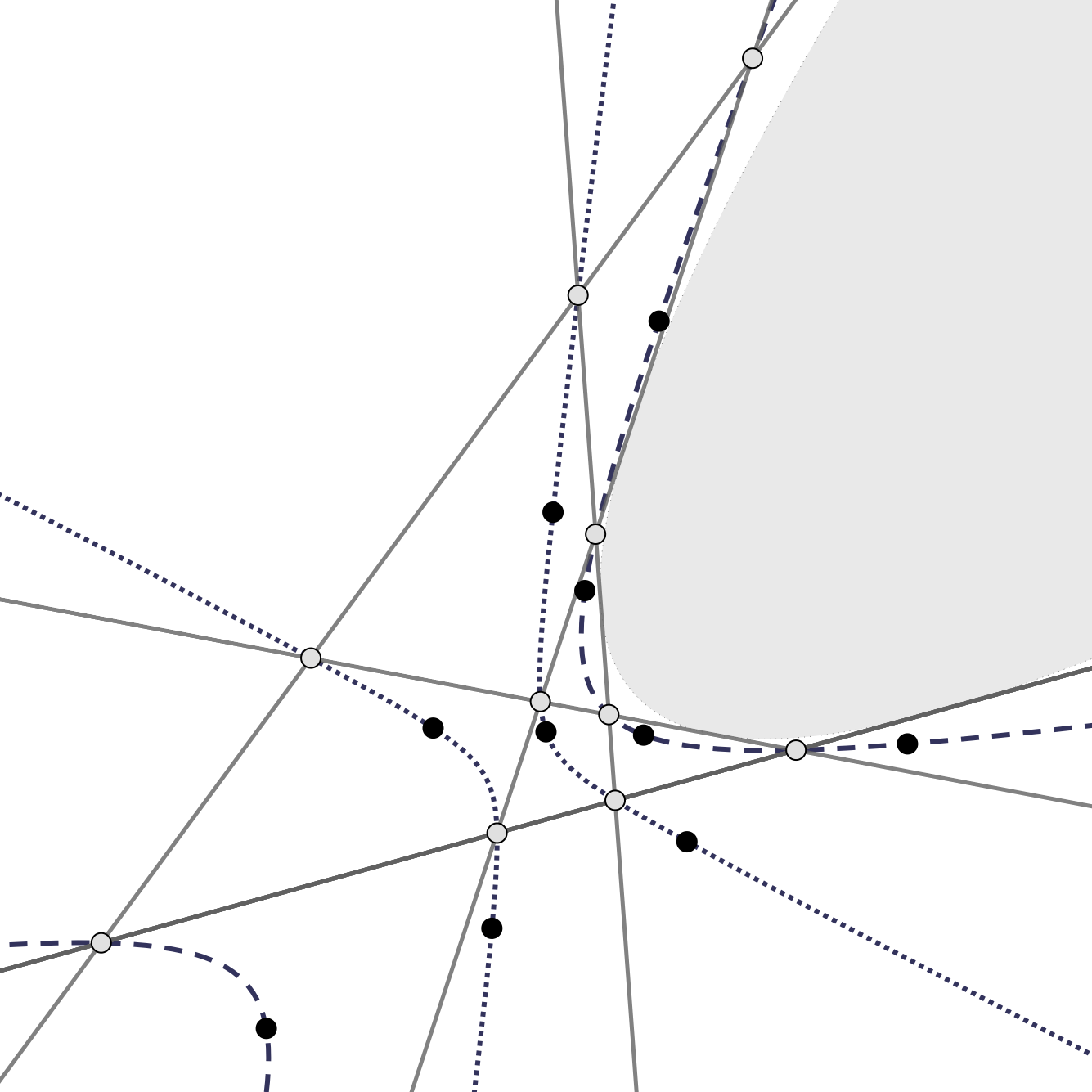}
		\label{fig:not-irred-1}}
	\caption{Construction of the set of points $X\cup Y$ on a reduced quartic.}
	\label{fig:globfig}
\end{figure}

\end{example}

In order to show that the condition $s=r$ or $s=r-1$ in  Theorem \ref{t.main} is also necessary for $X\cup Y$ to be a complete intersection, we prove the following lemma which holds with more general assumptions.
\begin{lemma}\label{l. necessary}
	Let $X$ and $Y$ be two disjoint star configurations defined by $r$ and $s$ lines, ($r\ge s$), respectively. If $X\cup Y$ is a complete intersection, then either $s=r$ or $s=r-1$.
\end{lemma}
\begin{proof}
	Since $h_{X}=(1,2,3, \ldots, r-1 )$ and $h_{Y}=(1,2,3, \ldots, s-1 ),$ 
	thus, by liaison, see formula~\eqref{eq.Corollary 5.2.19}, the $h$-vector of $X\cup Y$ must be
	$$ h_{X\cup Y}=(1,2,3, \ldots, r-1, s-1 , \ldots, 3,2,1  ).$$
	Since $X\cup Y$ is a complete intersection, then, to ensure the symmetry, we get $s=r$ or $s=r-1$.
\end{proof}

Theorem \ref{t.main} together with  Lemma \ref{l. necessary} give the following result. 
\begin{theorem}Let $X=\mathbb{S}(\ell_1,\ldots,\ell_r)$ and $Y=\mathbb{S}(m_1,\ldots,m_s)$ be two contact star configurations on the same conic.
	Then $X\cup Y$ is a complete intersection if and only if either $s=r$ or $s=r-1$.
\end{theorem}

\section{The $h$-vector of union of contact star configurations in $\mathbb P^2$}\label{s.other unions}
In this section we work in $\mathbb P^2$, so $\gamma$ is always an irreducible conic, and we set $S=\mathbb C[x,y,z]=\mathbb C[\mathbb P^2]$.
The main result of this section, Theorem \ref{t. h-vector 2 star }, shows that a union of two contact star configurations on a conic $\gamma$ has 
the same $h$-vector of a scheme of two fat points in $\mathbb P^2$. This point of view will allows us to make further considerations on the $h$-vector of more than two contact star configurations on a conic $\gamma$, see Theorem \ref{t. 3 fat 3 had}.

The $h$-vector of two fat points is well known, several papers investigate it in a more general setting, see for instance \cite[Theorem 1.5 and Example 1.6]{fatabbi2001} and also \cite{ catalisano1991fat,dg1984, harbourne_1998, harbourne_2000} just to cite some of them. Recall that if the multiplicities of the two points are $m$ and $n$, with  $m\ge n$, then the $h$-vector is
\begin{equation}\label{eq. 2 fat points}
 (1,2, \ldots, m-1, m, n, n-1, \ldots, 2, 1).
\end{equation}

In order to prove the claimed result, we need the following lemma.

\begin{lemma}\label{l. h vector difference of two star}
	Let $Y=\mathbb S(m_1, \ldots, m_s)$ be a star configurations of $s$ lines, and let $m_{s+1}, \ldots, m_{s+t}$ be $t$ further lines,  $t\ge 1$. 	Set $X=\mathbb S(m_1, \ldots, m_s, m_{s+1}, \ldots, m_{s+t})$.
 Then, the $h$-vector of $X\setminus Y$ is
	$$h_{X\setminus Y}=(1, \ldots, t-1,  \underbrace{t , t, \ldots,  t}_{s}).$$
\end{lemma}
\begin{proof}
	Since the $h$-vector of $X$ is 	$(1,2, \ldots, s+t-1)$ and $X\setminus Y$ is contained in a curve of degree~$t$, that is $m_{t+1}\cup \cdots\cup m_{t+s}$, then 
	$h_{X\setminus Y}\le(1, \ldots, t-1,  \underbrace{t , t, \ldots,  t}_{s})\le h_X.$
	Since, 
	$$|X\setminus Y|=\binom{s+t}{2}- \binom{s}{2} =\binom{t}{2}+st=1+ \cdots+ t-1+  \underbrace{t+ t+ \cdots+  t}_{s},	$$
	we are done.   
\end{proof}

The next example shows how we compute the $h$-vector of a union of two  contact star configurations in $\mathbb P^2$.
\begin{example}\label{ex.3-7}
	Let $X=\mathbb S( \ell_1,\ldots,\ell_7 )$ and $Y=\mathbb S( m_1,m_2,m_3 )$ be two  contact star configurations on a conic $\gamma$. Figure \ref{fig:3-7-base} gives a representation of this case. 

	 	
	In order to compute the $h$-vector of $X\cup Y$ we consider three further lines $m_4,m_5,m_6$ tangent to $\gamma.$ Let $Y'=\mathbb S({m_1,\ldots,m_6})\supseteq Y$,  see  Figure \ref{fig:3-7}.
	
		\begin{figure}[h]
	\centering
	\subfloat[The set $X\cup Y$.]{
		\includegraphics[width=.42\linewidth]{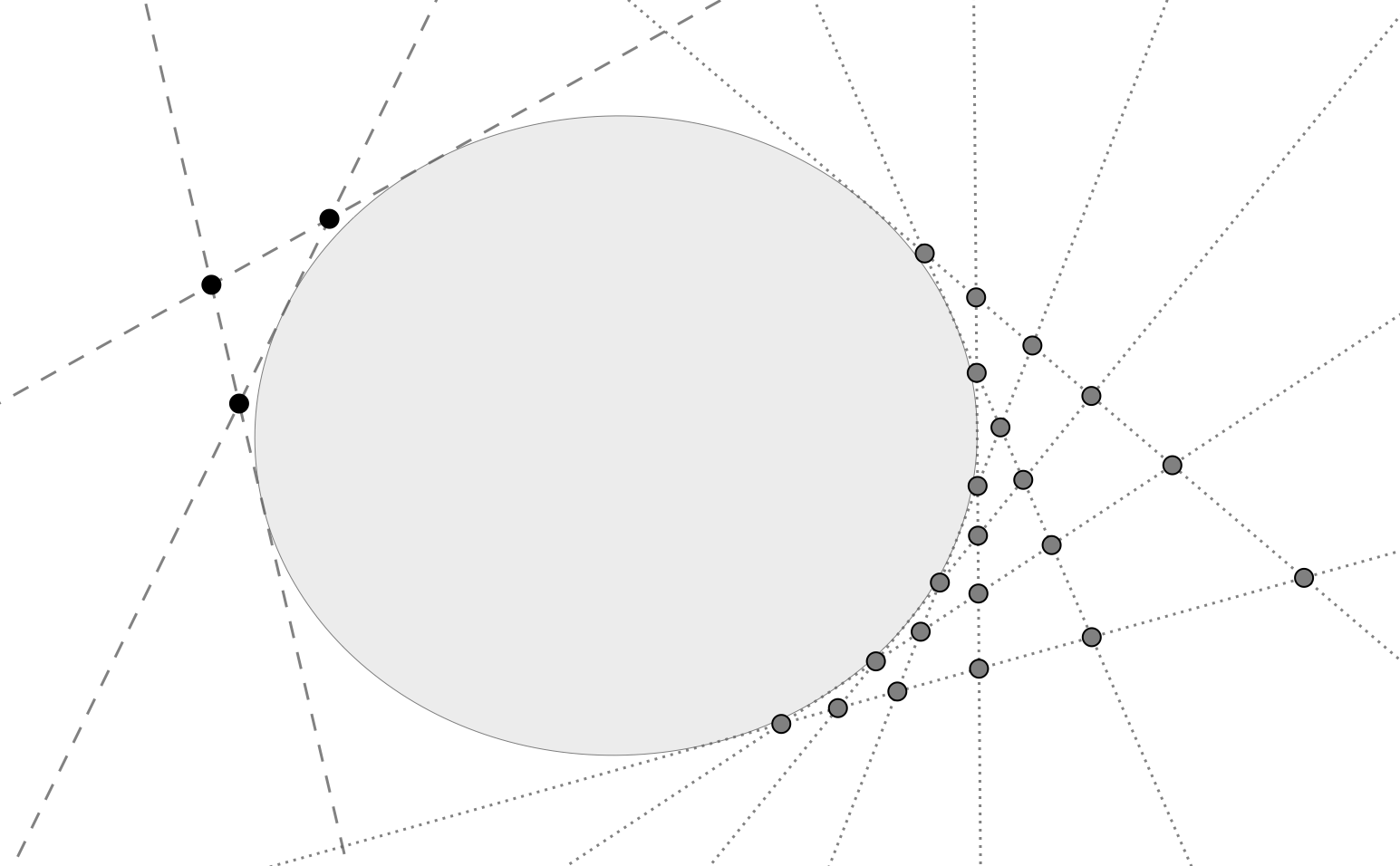}			\label{fig:3-7-base}}		\qquad
	\subfloat[The set $X\cup Y'$.]{
		\includegraphics[width=.49\linewidth]{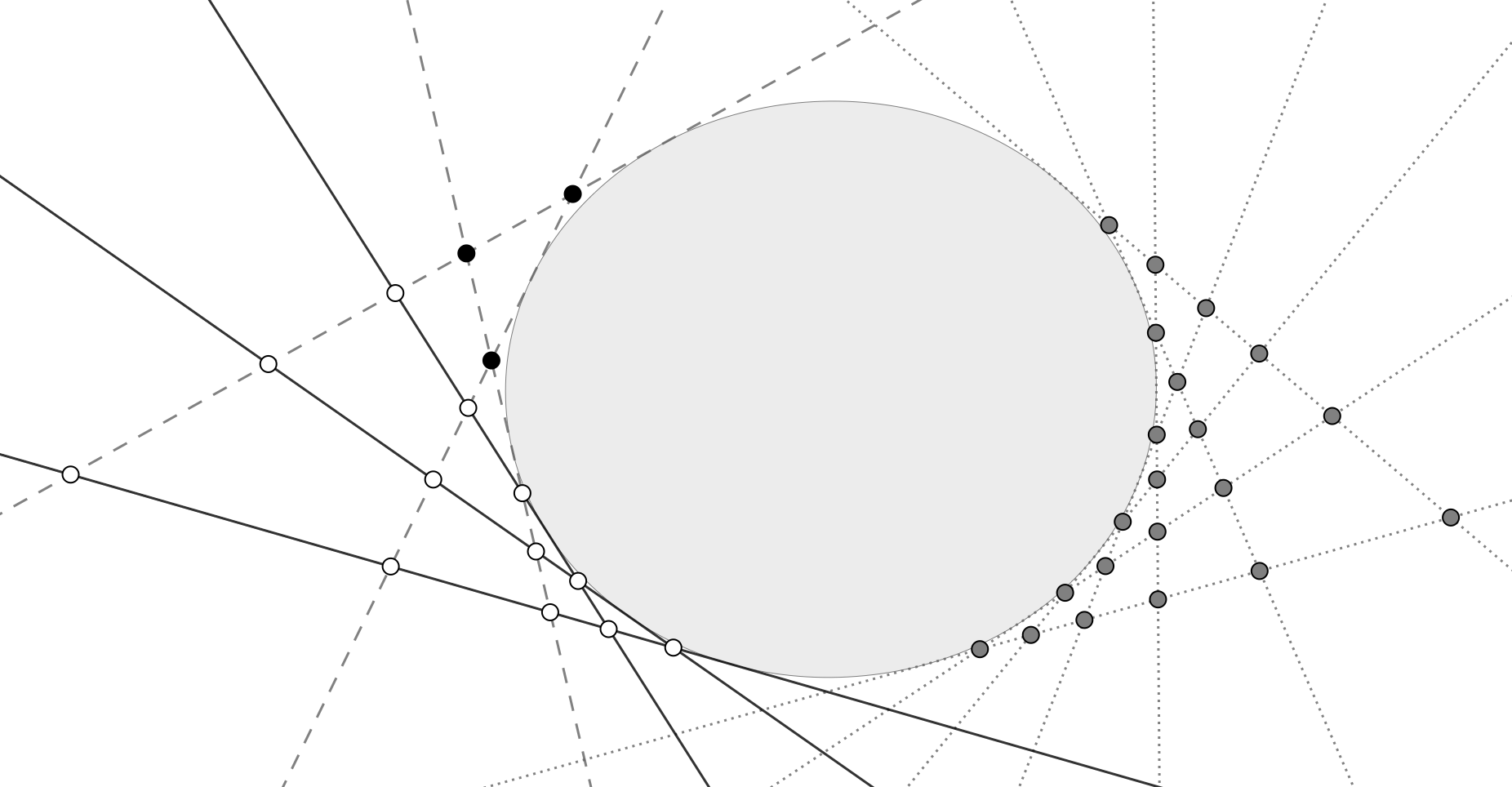}
		\label{fig:3-7}}
	\caption{ Example \ref{ex.3-7}.}
	\label{fig:ex.3-7}
\end{figure}

	From Theorem \ref{t.main}  (b), $X\cup Y'$ is a complete intersection of type $(6,6)$, hence the $h$-vector of $X\cup Y'$ is $h_{X\cup Y'}=(1,2,3,4,5,6,5,4,3,2,1)$. By Lemma \ref{l. necessary}, the $h$-vector of $Y'\setminus Y$ is  $h_{Y'\setminus Y}=(1,2,3,3,3)$. 
Thus, by formula \eqref{eq.Corollary 5.2.19}, we get
\[	\begin{array}{l|cccccccccccccccccc}
	t& 0& 1 & 2 & 3 & 4 & 5 & 6 & 7& 8& 9& 10\\
	\hline
	h_{X\cup Y'}(t) & 1 & 2 & 3& 4& 5& 6& 5& 4& 3 & 2& 1\\
	h_{Y'\setminus Y}(10-t) &   &   &  &  &  &  & 3& 3& 3 & 2& 1  \\
	h_{X\cup Y}(t) & 1 & 2 & 3& 4& 5& 6& 2& 1\\
\end{array}
\]	
which shows that  $X\cup Y$ has the same $h$-vector of a scheme of two  fat points of multiplicity $2$ and $6$, see  \eqref{eq. 2 fat points}.
\end{example}

Now we state the  general theorem.
\begin{theorem}\label{t. h-vector 2 star }
	Let $X=\mathbb{S}(\ell_1,\ldots,\ell_r)$ and $Y=\mathbb{S}(m_1,\ldots,m_s)$ be two contact star configurations on a conic.
 Let $P, Q$ be two distinct points in $\mathbb P^2.$ Then
    \[
    h_{X\cup Y} = h_{(r-1)P+(s-1)Q}. 
    \]
\end{theorem}
\begin{proof}
Assume $s\le r$. If either $s=r$ or $s=r-1$, then the statement follows from Theorem \ref{t.main}. In fact, in both the cases, $X\cup Y$ is a complete intersection and it has the required $h$-vector. 

So, assume $s<r-1$ and let $t=r-s-1$. Consider $t$ further lines, $m_{s+1},\ldots,m_{s+t}$,  tangent to $\gamma$, and denote by $Y'=\mathbb S(m_1, \ldots , m_{s+t})$. 
From Theorem \ref{t.main} (b),   $X\cup Y'$ is a complete intersection of type $(t+s,t+s)$ and its $h$-vector is
\[
h_{X\cup Y'}=(1,2,\ldots, t+s-1,t+s, t+s-1, \ldots, 2,1 ).
\] 

By Lemma \ref{l. necessary}, the $h$-vector of $Y'\setminus Y$  is  
$$h_Z=(1,2,\ldots, t-1, \underbrace{t,t, \ldots, t}_{s}).$$
By formula \eqref{eq.Corollary 5.2.19}, we get 
$$\begin{array}{lll}
h_{X\cup Y}&=&(1,2,\ldots, t+s-1,t+s, s-1, s-2, \ldots, 2, 1)=\\
&=&(1,2,\ldots, r-2,r-1, s-1, s-2, \ldots, 2, 1)
\end{array}$$
which is the $h$-vector of $(r-1)P+(s-1)Q$. 
\end{proof}

Theorem \ref{t.main}  allows us to compute the $h$-vector of a union of three contact star configurations, on a same conic in the special case described in Theorem  \ref{t. 3 fat 3 had}. The proof of Theorem  \ref{t. 3 fat 3 had} requires the following well known result about the Hilbert function of a scheme of three fat points not lying  on a line  (see for instance \cite[Theorem 3.1]{catalisano1991fat}). 
\begin{proposition}[The Hilbert function of three fat points]\label{p. 3 fat}
	Let $Z=m_1P_1+m_2P_2+m_3P_3$ be a scheme of three general fat points of multiplicity $m_1\ge m_2\ge m_3\ge 1$, respectively. Then 
	\[
	H_Z(d)=\left\{\begin{array}{lll}
	d+1+H_Z'(d-1)& if & 0 \le  d \le m_1+m_2-2\\ 
	\deg(Z) & if & d \ge m_1+m_2-1\\
	\end{array}\right.
	\]
	where $Z'=(m_1-1)P_1+(m_2-1)P_2+m_3P_3.$ 
\end{proposition}

\begin{theorem}\label{t. 3 fat 3 had}Let $X=\mathbb S(\ell_1 \ldots, \ell_{r})$, $Y=\mathbb S(m_1, \ldots, m_s)$ and $W=\mathbb S(n_1, \ldots, n_{t})$ be contact star configurations on a conic, where $t\ge r\ge s\ge 2$.
	Let $Z= (t-1)P_1+(r-1)P_2+(s-1)P_3$ be a scheme of three general fat points.
	If $t\in \{ r+s-1, r+s, r+s+1 \},$ then 
	$$h_{X\cup Y \cup W} = h_Z$$
	\begin{equation}\label{eq. h vec Z}
	=(1,\ldots, t-2,\ t-1, \underbrace{r+s-2,\ r+s-4,  \ldots, r-s+2}_{ s -1},\ \underbrace{r-s, r-s-1, r-s-2, \ldots }_{r-s}).
	\end{equation}
\end{theorem}
\begin{proof}	
	We prove the theorem for $t=r+s.$ The other two
	cases can be proved similarly.
	
	First we compute the $h$-vector of $X\cup Y \cup W$.  
	
	Note that $X\cup Y$ is contained in the contact star configuration $T= \mathbb S(m_1,\ldots, m_r, n_1, \ldots, n_s)$. 
	Then $ T\cup W$, by Theorem \ref{t.main} (d), is a complete intersection of type $(r+s-1,r+s),$ and so
	$$h_{T\cup W}=(1,\ldots, r+s-2,\ r+s-1,\ r+s-1,\ r+s-2, \ldots, 1).$$
	
	Since the set $T\setminus (X\cup Y)$  is a complete intersection of type $(r,s)$, then its $h$-vector is 
	$$(1,\ldots,s-1,\ \underbrace{s,\ \ldots, s}_{r-s+1}, s-1, \ldots, 1).$$
	
	Hence, by  formula \eqref{eq.Corollary 5.2.19} we get  

	\begin{equation}\label{hXYW} \begin{array}{c}
	h_{X\cup Y\cup W}\\
	=(1,\ldots, r+s-2,\ r+s-1, \underbrace{r+s-2,\ r+s-4,  \ldots, r-s+2}_{ s -1},\underbrace{ r-s, r-s-1, r-s-2, \ldots}_{r-s}).
	\end{array}
	\end{equation}
	Observe that, 
	for $s=2$  we get
	$$ h_{X\cup Y\cup W}=(1,\ldots, r,\ r+1,\ r,\underbrace{ r-2, r-3, \ldots}_{r-2}),$$
	for $r=s$ we have
	$$ h_{X\cup Y\cup W}=(1,\ldots, 2r-2,\ 2r-1,\ \underbrace{2r-2,\ 2r-4,  \ldots, 2}_{ r-1}),$$
	and for $r=s=2$ 
	$$ h_{X\cup Y\cup W}=(1, \ 2, \ 3,\ 2).$$


	We will prove the theorem by induction on $2r+2s-3$, that is, on the sum of the multiplicities of the three  fat points.
	If $ 2r+2s-3=5$, then 
	$Z= 3P_1+P_2+P_3$, whose $h$-vector is $(1, \ 2, \ 3,\ 2)$,
	so the statement is proved for $(r,s)=(2,2).$
	
	Assume $2r+2s-3\geq 7$, and recall that $Z= (r+s-1)P_1+(r-1)P_2+(s-1)P_3$ ( $r \geq s \geq 2$).  By Proposition \ref{p. 3 fat} we have

	$$H_Z(d)=\left\{\begin{array}{llc}
d+1+ H_{Z'}(d-1) & \text{if}\ 0\le d\le (r+s-1)+(r-1)-2\\
\deg(Z) & \text{if}\ d\ge (r+s-1)+(r-1)-1\\
\end{array}
\right.$$
where $Z'=(r+s-2)P_1+(r-2)P_2+(s-1)P_3$. Hence the $h$-vector of $Z$ is
	\begin{equation}\label{eq. h Z}
		h_Z(d)=\left\{\begin{array}{llc}
			1+ h_{Z'}(d-1) & \text{if}\ 0\le d\le (r+s-1)+(r-1)-2\\
			\deg(Z)-d-H_{Z'}(d-2) & \text{if}\ d= (r+s-1)+(r-1)-1\\
			0 & \text{if}\ d\ge (r+s-1)+(r-1)\\
		\end{array}.
		\right.
	\end{equation}
	
	Now we will compute  $H_{Z'}(d-2) $ for $d= (r+s-1)+(r-1)-1 =2r+s-3 $.
	
	For $r>s$ we have  $r+s-2> r-2 \geq s-1$, and $d-2 = 2r+s-5 = (r+s-2) +(r-2)-1$, which is the sum of the two highest multiplicities.  Hence, by Proposition \ref{p. 3 fat}, we get $H_{Z'}(2r+s-5) = \deg(Z')$. 
	
	If $r=s$, we have $Z'=(2r-2)P_1+(r-2)P_2+(r-1)P_3$ and we need to compute $H_{Z'}(3r-5)$. Since the line $P_1P_3$  is a fixed component for the curves of degree $3r-5$ through $Z'$, we have
	$$\dim (I_{Z'} )_{3r-5}=\dim (I_{Z''})_{3r-6}, $$
	where $Z''= (2r-3)P_1+(r-2)P_2+(r-2)P_3$. 
	Since the scheme $Z''$ gives independent conditions to the curve of degree  $3r-6$ (see again Proposition \ref{p. 3 fat}), hence   $H_{Z''}(3r-6)= \deg (Z'')$.
	It follows that
	\[
	\begin{array}{rcl}
	H_{Z'}(3r-5)&=& \binom{3r-3}{2}- \dim  (I_{Z'})_{3r-5}=\binom{3r-3}{2}- \dim  (I_{Z''})_{3r-6}=\binom{3r-3}{2}- \binom{3r-4}{2} + \deg(Z'')\\
& =& 3r^2-5r+1.
	\end{array}
	\]

	Thus, for $d= 2r+s-3$ we have
	$$\deg(Z)-d-H_{Z'}(d-2) =
	\left\{\begin{array}{lccc}
	\deg(Z)-(2r+s-3)-\deg(Z')= 1  & \text{if}\ r > s\\
	\deg(Z) -(3r-3)-(3r^2-5r+1)= 2  & \text{if}\ r=s\\
	\end{array}\right.
	.$$
	From this equality and from (\ref {eq. h Z}) we get 

	\begin{equation} \label {hzd}
		h_Z(d)=\left\{\begin{array}{llc}
			1+ h_{Z'}(d-1) & \text{if}\ 0\le d\le 2r+s-4\\
			1 & \text{if}\ d= 2r+s-3 \ \text{and} \ r> s \\
			2 & \text{if}\ d=3r-3 \  \text{and} \ r=s \\
			0 & \text{if}\ d\ge 2r+s-2\\
		\end{array}.
		\right. 
	\end{equation}
	
	Now, if $r>s$, from  the inductive hypothesis, by substituting $r$ with $r-1$ in formula
	\eqref{eq. h vec Z}, we get the $h$-vector of $Z' = (r+s-2)P_1+(r-2)P_2+(s-1)P_3$, that is,
	$$ h_{Z'}=(1,\ldots, r+s-3,\ r+s-2, \underbrace{r+s-3,\ r+s-5,  \ldots, r-s+1}_{ s -1},\underbrace{ r-s-1, r-s-2, \ldots)}_{r-s-1}.$$
	In case $r=s$, we have $Z' =  (2r-2)P_1+(r-1)P_3+(r-2)P_2$, (note that $2r-2\ge r-1\ge r-2$). Let $s'=r-1$. With this notation $Z'=(r+s'-1)P_1+(r-1)P_3+(s'-1)P_2$. By applying the inductive hypothesis and then by substituting $s'$ with $r-1$, we get 
$$	h_{Z'}=(1,\ldots, r+s'-2,\ r+s'-1, \underbrace{r+s'-2,\ r+s'-4,  \ldots, r-s'+2}_{ s' -1}, r-s', r-s'-1, r-s'-2, \ldots, 1)=$$
$$=(1,\ldots, 2r-3,\ 2r-2, \underbrace{2r-3,\ 2r-5,  \ldots, 3}_{ r-2},  1).$$
%
%
	
	By (\ref {hXYW}) and (\ref {hzd}) the conclusion follows.
\end{proof}

	\begin{remark}
	Note that this theorem gives a non-algorithmic formula for the $h$-vector of three fat points of multiplicities $m_1, m_2, m_3$ when $m_1=m_2+m_3$ or $m_1=m_2+m_3 \pm 1$. 
\end{remark}

 We illustrate the case $r=3, s=3, t=6$ in the following example.   
\begin{example}\label{ex.3-3-6}
	Let $X=\mathbb S(\ell_1, \ell_2, \ell_3)$, $Y=\mathbb S(m_1, m_2, m_3)$ and $W=\mathbb S(n_1, \ldots, n_6)$ be contact star configurations on the same conic, see Figure \ref{fig:3-3-6}. 
	Set $T=\mathbb S(\ell_1, \ell_2, \ell_3, m_1,m_2, m_3)$, see Figure  \ref{fig:3-3-6grid}.
	
		\begin{figure}[h]
		\centering
		\subfloat[The union of three contact star configurations on a conic.]{
		\includegraphics[width=.44\linewidth]{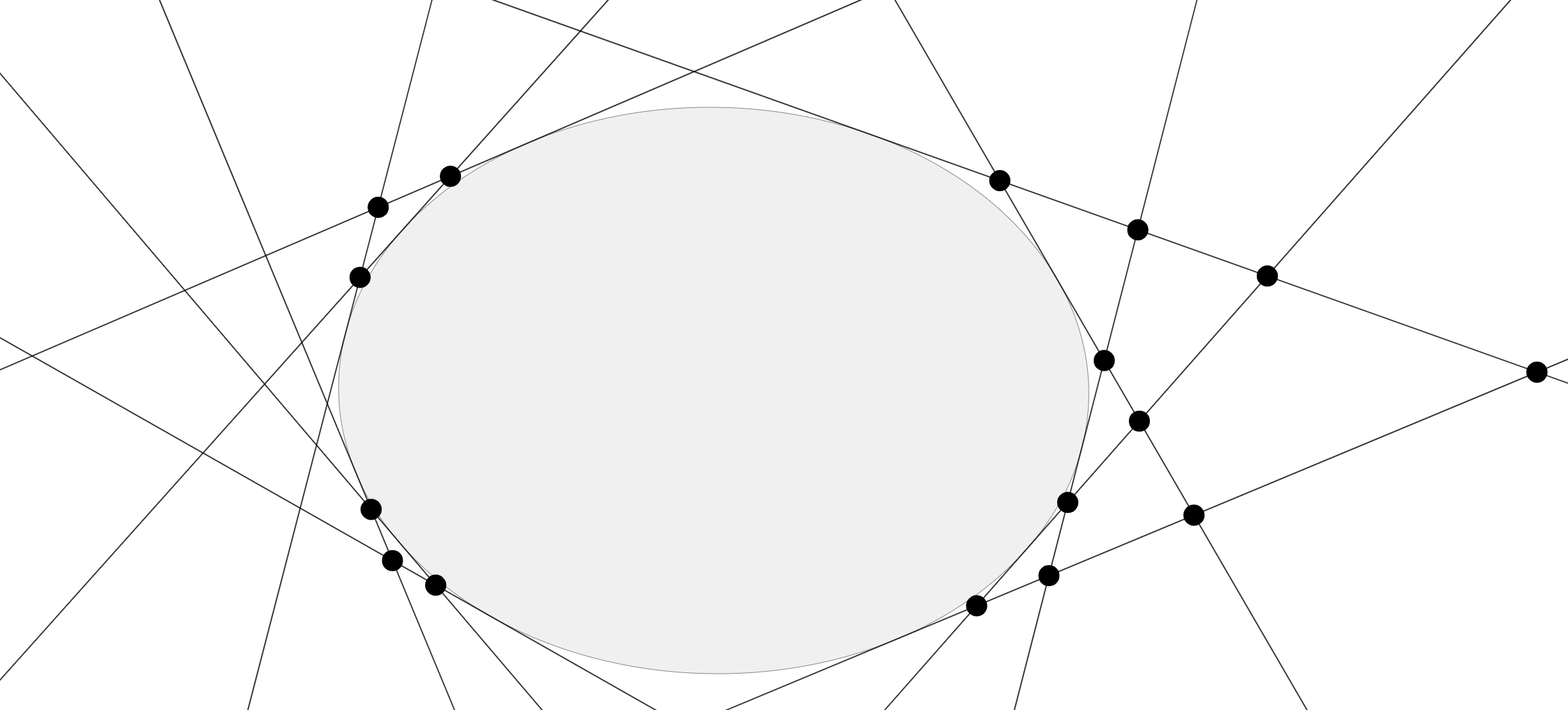}			\label{fig:3-3-6}}		\qquad
		\subfloat[The nine points in $T\setminus (X\cup Y)$ on a grid.]{
				\includegraphics[width=.44\linewidth]{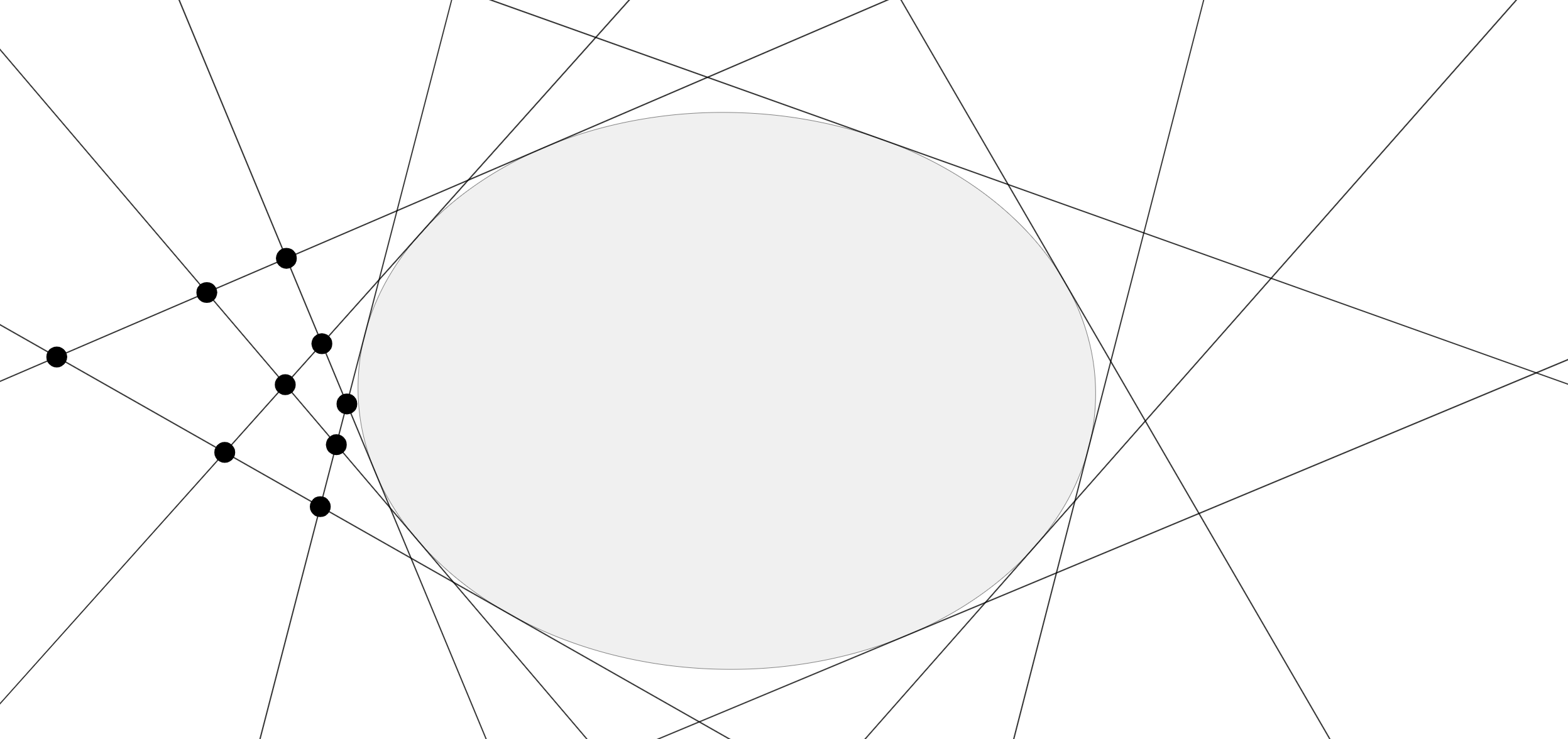}
					\label{fig:3-3-6grid}}
		\caption{ Example \ref{ex.3-3-6}.}
		\label{fig:ex.3-3-6}
	\end{figure}
	
	Then  
	\[\begin{array}{lcl}
	h_{W\cup T}&=& (1,2,3,4,5,5,4,3,2,1),\\
	h_{T\setminus (X\cup Y)}&=&(1,2,3,2,1),
	\end{array}\]
	and, by liaison, 	$h_{X\cup Y\cup W}=(1,2,3,4,5,4,2)$,
	which is the $h$-vector of three fat points of multiplicity $2,2,5$.
\end{example}

Theorem \ref{t. 3 fat 3 had} and experiments using CoCoA~\cite{CoCoA} suggest the following conjecture.

\begin{conjecture}\label{conj. 3,4 star} 
	The $h$-vector of a scheme of $s\le 4$ general fat points  of multiplicities $m_i$, $(i=1,\ldots, s)$ is equal to the $h$-vector of the union of $s$ contact star configurations defined by $m_i+1$ lines tangent to the same conic.  
\end{conjecture}

The next example shows that the conjecture does not hold for $s=5$.

\begin{example}\label{e. five double points}
The $h$-vector of five general fat points of multiplicity 2 in $\mathbb P^2$ is
$$(1,2,3,4,4,1)$$
but we checked  with  CoCoA \cite{CoCoA} that the $h$-vector of five general contact star configurations on a conic defined by three lines is
$$(1,2,3,4,5).$$ %
\end{example}

\section{Applications of Theorem \ref{t.main} to polygons}\label{s.applications}

As an application of Theorem \ref{t.main} we get a result that in a certain sense extend the Brianchon's Theorem to an octagon circumscribed to a conic. 
\begin{proposition}\label{p.brianchonlike}
	Let $A_1,\ldots, A_8$ be the vertices of an octagon and let $\ell_{ij}$ be the line $A_i A_j$.
	Set $P_1=\ell_{18}\cap\ell_{23} $ and $P_2=\ell_{12}\cap\ell_{34}$ and let $\gamma_1$ and $\gamma_2$ be the conics  through $A_1,A_2, A_5,A_6,P_1$ and $A_2,A_3, A_6,A_7,P_2$, respectively. 
	Let $\gamma_1\cap \gamma_2=\{A_2, A_6, B_1, B_2\}$.  If
	the octagon is circumscribed to a conic $\gamma$, then the points $A_4, A_8, B_1,B_2$   are on a line. (See Figure \ref{fig.brianchon2})\begin{figure}[h]
		\centering
		\includegraphics[scale=0.45]{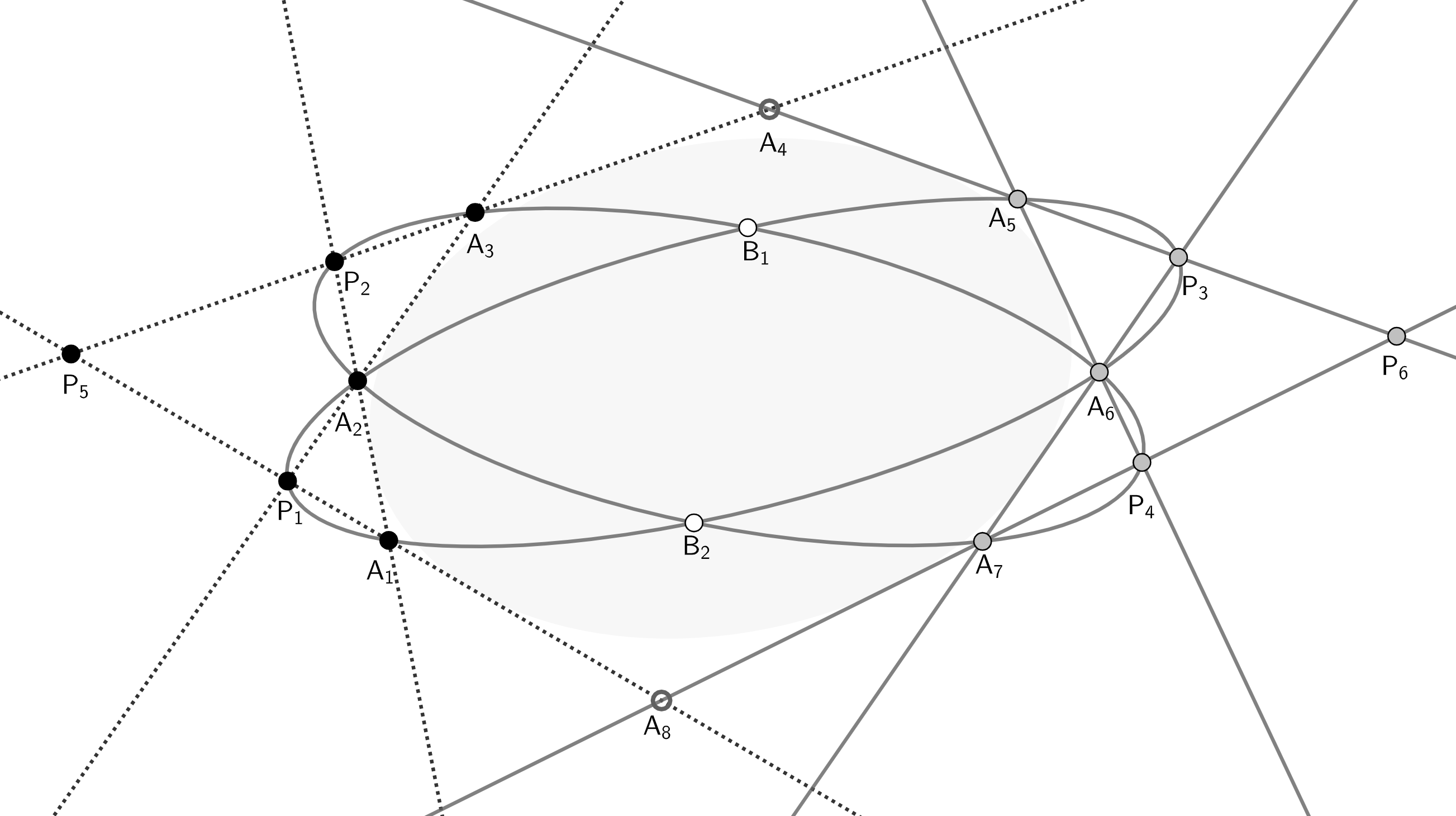}
		\caption{The case described in Proposition \ref{p.brianchonlike} }\label{fig.brianchon2}
\end{figure}\end{proposition}

\begin{proof}	
	Note that, by Theorem \ref{t.main} (d), the point $P_3=\ell_{45}\cap \ell_{67}$ belongs to the conic $\gamma_1$. In fact, $A_1, A_2, P_1$ and $A_5, A_6, P_3$ are two contact star configurations on the same conic and then a complete intersection of type $(2,3).$  Analogously, the point $P_4=\ell_{56}\cap \ell_{78}$ belongs to the conic $\gamma_2$.
	Moreover, observe that the points $A_1,A_2,A_3,P_1,P_2,P_5=\ell_{18}\cap\ell_{34}$, and the points $A_5,A_6,A_7,P_3,P_4,P_6=\ell_{45}\cap\ell_{78}$ are two  contact star configurations each defined by four lines tangent to the same conic, hence by Theorem \ref{t.main} (d) these twelve points are a complete intersection of type $(3,4),$ thus their $h$-vector is $(1,2,3,3,2,1).$
	
	\noindent Now consider the two quartics $\gamma_1\cup\ell_{34}\cup\ell_{78}$ and $\gamma_2\cup\ell_{18}\cup\ell_{45}$. This two quartics meet in a complete intersection  of sixteen points, that consists of the twelve points described above and the points $A_4, A_8, B_1, B_2$. 
	By  relation \eqref{eq.Corollary 5.2.19}, we get 
	\[
	\begin{array}{l|ccccccccc}
	& 0& 1 & 2 & 3 & 4 & 5 & 6\\
	\hline
	\text{the $h$-vector of the sixteen points } & 1 & 2 & 3& 4& 3 & 2& 1\\
	\text{the $h$-vector of the twelve points } & 1 & 2 & 3& 3 & 2 & 1 \\
	\text{the $h$-vector of} \ \{A_4, A_8, B_1, B_2\} &  &    &  & 1& 1& 1& 1\\
	\end{array}
	\]
	the table above shows that the four points are collinear.
\end{proof}

\begin{corollary}\label{c.conic8points}
		Let $A_1,\ldots, A_8$ be the vertices of an octagon and let $\ell_{ij}$ be the line $A_i A_j$.
	Set $P_1=\ell_{18}\cap\ell_{23} , P_2=\ell_{12}\cap\ell_{34}, P_3=\ell_{23}\cap\ell_{45}, P_4=\ell_{34}\cap\ell_{56} $ and let $\gamma_i$ be the conic  through $A_i,A_{i+1}, A_{4+i},A_{5+i},P_i$ , $i=1,\ldots 4$ ($A_9=A_1$). 
	If	the octagon is circumscribed to a conic $\gamma$, then the eight points $(\gamma_1\cap \gamma_3)\cup (\gamma_2\cap \gamma_4)$   are on a conic. (See Figure \ref{fig.conic8points}).
\end{corollary}
\begin{proof}
		\begin{figure}[h]
		\centering
		\includegraphics[scale=0.30]{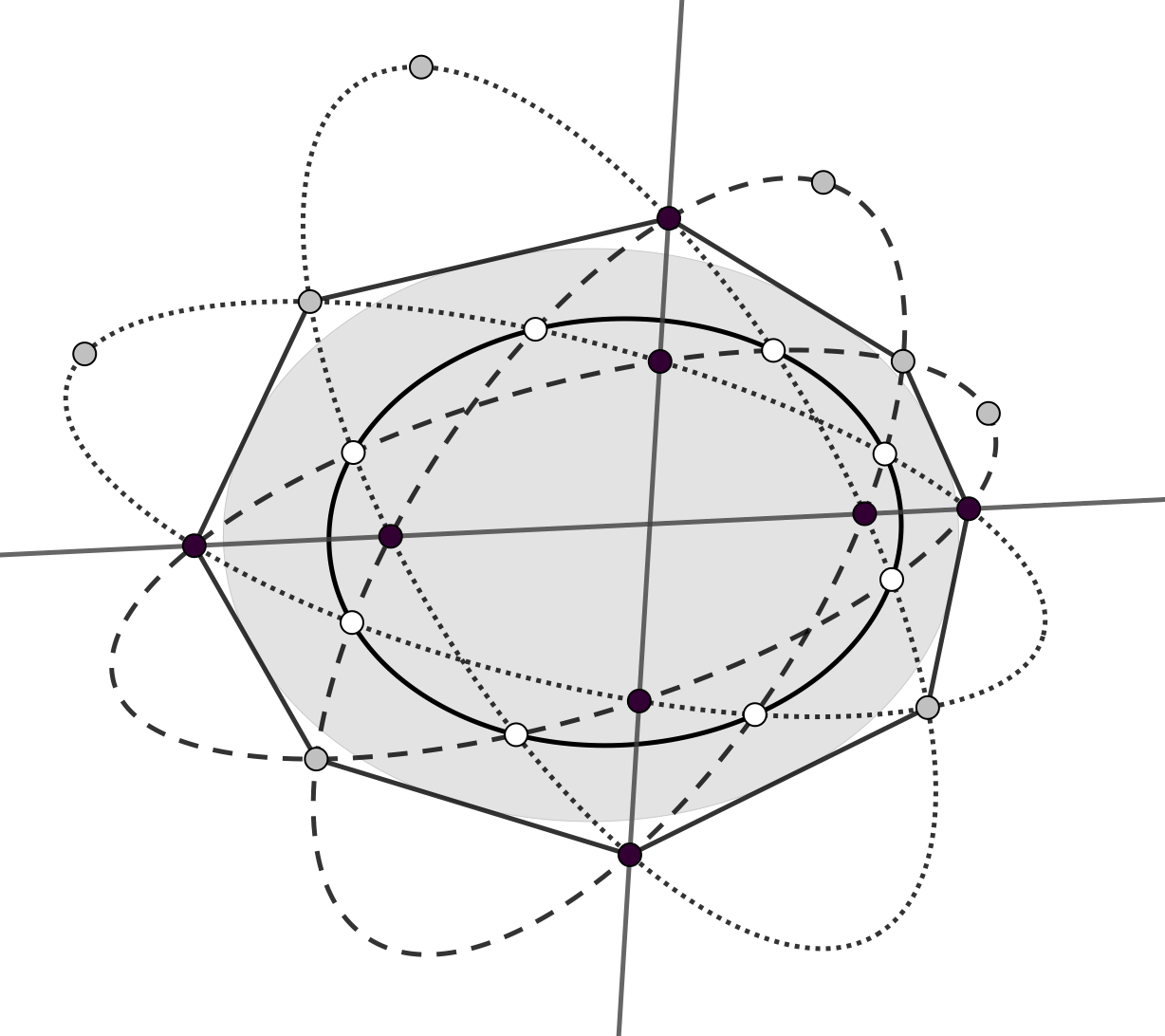}
		\caption{The case described in Corollary \ref{c.conic8points} }\label{fig.conic8points}\end{figure}
	The quartics $\gamma_1\cup \gamma_4$ and $\gamma_2\cup \gamma_3$ meet in sixteen points, in black and white in figure \ref{c.conic8points}. Eight of them, the black dots, that is $(\gamma_1\cap \gamma_2)\cup(\gamma_3\cap \gamma_4)$,  by Proposition \ref{p.brianchonlike}, lie on two lines. So, the $h$-vector  of these eight points must be $(1,2,2,2,1)$. By formula \eqref{eq.Corollary 5.2.19} the residual points lie on a conic.
\end{proof}

Recall that, by Theorem \ref{t.main}(e), the six points union of two contact star configurations each defined by three line tangent to the same conic, are contained on a conic. An interesting case occurs considering three such contact star configurations, see Proposition \ref{p.3conics}, it can also be translated in a property of a polygon of nine sides circumscribed to a conic.

\begin{proposition}\label{p.3conics}	
	Let $X_{1}=\mathbb{S}(\ell_1, \ell_2, \ell_3), X_{2}=\mathbb{S}(m_1, m_2, m_3),X_{3}=\mathbb{S}(n_1, n_2, n_3)$ be contact star configurations on a conic.
	Let $\gamma_{ij}$ be the conic containing $X_i\cup X_j$. Then,  $\gamma_{12},\gamma_{13},\gamma_{23}$ meet in a point.
 	See Figure \ref{fig:3conics1}. 
\end{proposition}
\begin{proof}
	Consider the cubic $\gamma_{12}\cup n_1$ and the quartic $\gamma_{13}\cup m_1\cup m_2$ (respectively  dotted and dashed in Figure \ref{fig:3conics2}).
	The cubic and the quartic meet in twelve points (twice in $m_1\cap m_2$). Six of these  points, precisely $X_{1}\cup \mathbb S(n_1, m_1, m_2 )$, lie on a conic, by Theorem \ref{t.main}(e).  So, also the residual six points lie on a conic (by formula \ref{eq.Corollary 5.2.19}) that is, by Theorem \ref{t.main}(d), the conic $\gamma_{23}.$	
\begin{figure}[h]
	\centering
	\subfloat[Three conics meeting in a point.]{
		\includegraphics[width=.35\linewidth]{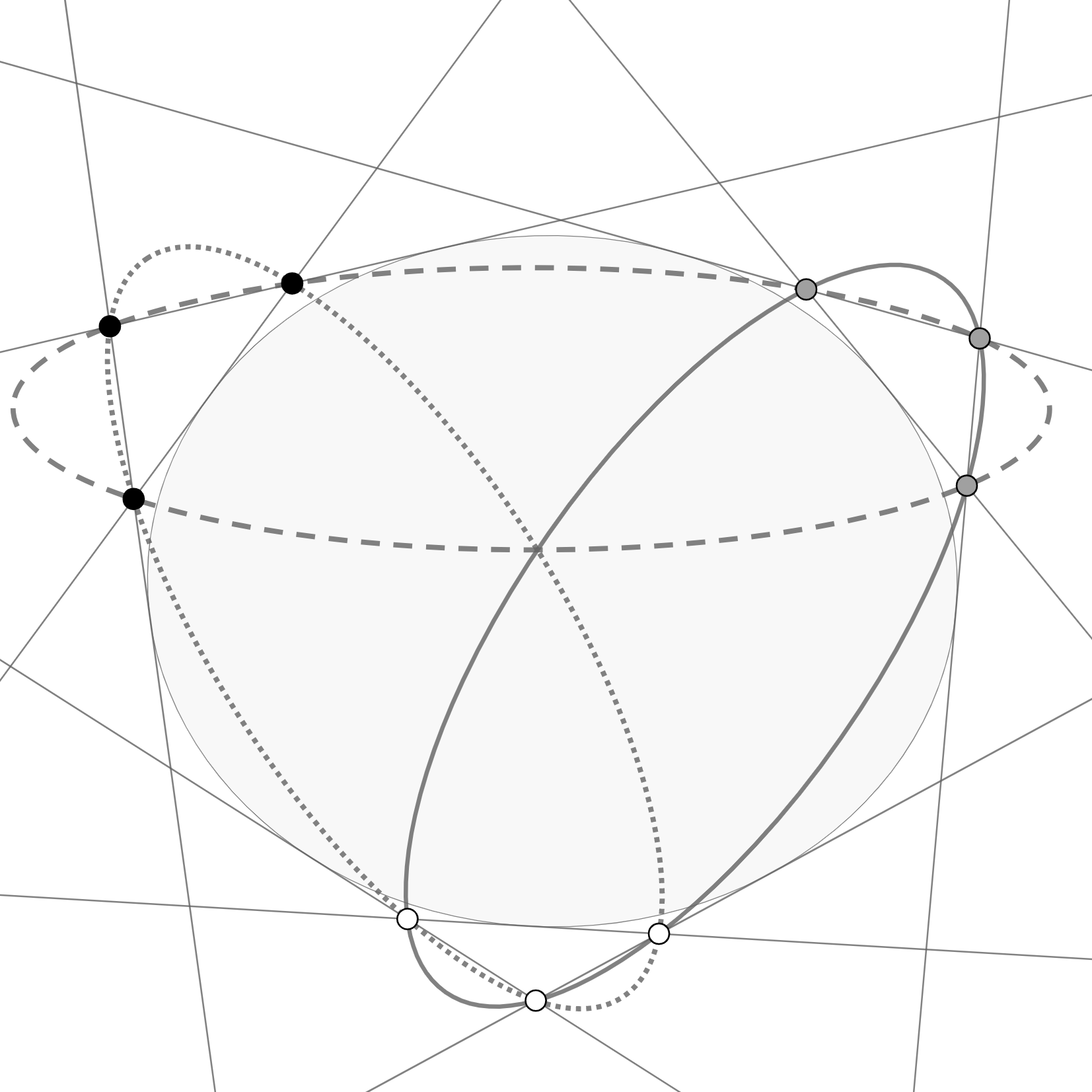}		\label{fig:3conics1}}		\qquad
	\subfloat[The cubic and the quartic.]{
		\includegraphics[width=.40\linewidth]{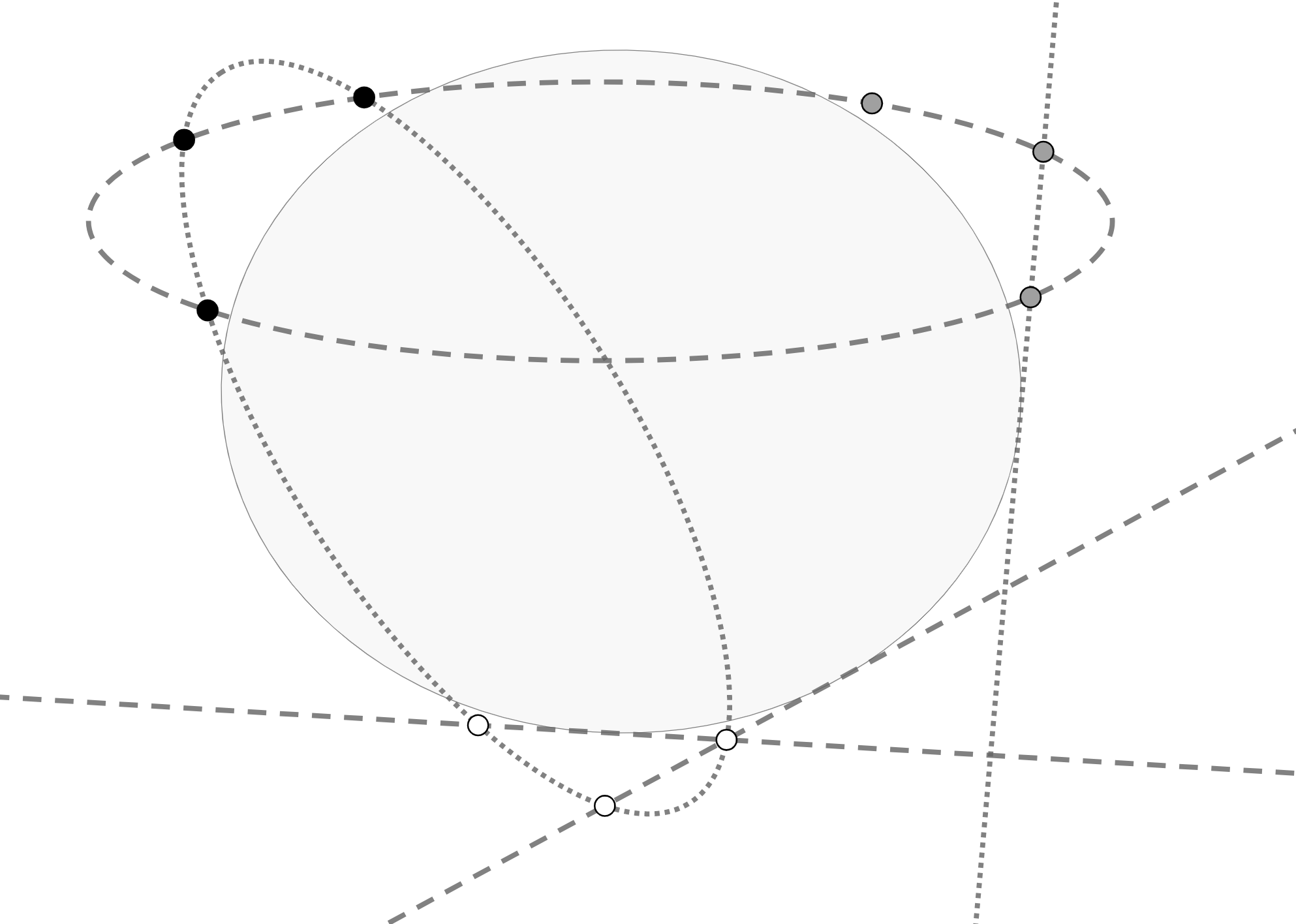}
		\label{fig:3conics2}}
	\caption{ Proposition  \ref{p.3conics}.}
	\label{fig:3conics}
\end{figure}
\end{proof}

\section{Further directions}\label{s. future}

According to our computations, it should be possible to extend some of the results in this paper to higher dimensional spaces. We state the following conjecture.
\begin{conjecture}\label{conj rnc}
	Let $X:=\mathbb{S}(\ell_1,\ldots, \ell_r)$ and $Y:=\mathbb{S}(m_1,\ldots, m_s)$  be two contact star configurations in $\mathbb P^n$, where $r\ge s$ and all the hyperplanes are  distinct. Then 
 the $h$-vector of $X\cup Y$ is $$h_{X\cup Y}=\left(1, \binom{n}{n-1}, \ldots, \binom{r-1}{n-1}, \binom{s-1}{n-1}, \ldots, \binom{n}{n-1}, 1\right).$$
	In particular, if either $s=r$ or $s=r-1$ then $X\cup Y$ is a Gorenstein set of points. 

\end{conjecture}

If Conjecture \ref{conj rnc} is true then, it is possible to generalize Remark \ref{r. c i had} in order to construct Gorenstein sets of rational points in $\mathbb P^n$ with a special $h$-vector. We show the procedure in the following example.

\begin{example}
	Let $S=\mathbb C[x,y,z,t]=\mathbb C[\mathbb P^3]$ and
	let $\ell\subseteq \mathbb P^3$ be the line in $\mathbb P^3$ defined by the ideal $(z-x-y, t-x+y).$
	Consider two sets of four points on $\ell$
	$$\begin{array}{rcl}
	X&=&\{[1, 1, 2, 0], [1, 2, 3, 1], [1, 3, 4, 2], [1, 4, 5, 3]\},\\
	Y&=&\{[1, -1, 0, -2], [1, -2, -1, -3], [1, -3, -2, -4], [1, -4, -3, -5]\}.\\
	\end{array}$$
	
	Then
	$$\begin{array}{rcl}
	\mathbb{X}^{^{\underline{\star} 3}}\cup \mathbb{Y}^{^{\underline{\star} 3}}&=&\{[1, 6, 24, 0], [1, 8, 30, 0], [1, 12, 40, 0], [1, 24, 60, 6]\} \cup \\
	&& \{[1, -6, 0, -24], [1, -8, 0, -30], [1, -12, 0, -40], [1, -24, -6, -60]\}.\\
	\end{array}$$
	
	According to CoCoA, the set of eight points ${\mathbb{X}}^{^{\underline{\star} 3}}\cup \mathbb{Y}^{^{\underline{\star} 3}}\subseteq \mathbb P^3$ is in fact Gorenstein  and its $h$-vector is $(1,3,3,1).$
\end{example}

On the other hand, it is interesting to ask if contact star configurations need to be constructed on rational normal curves. Of course, one can extend the definition by taking, for instance, high contact linear spaces to some other irreducible curve or surface and, with some assumptions of generality, get again a star configuration. 
However, we don't know if this construction on a variety of different kind will lead to configurations with special properties either from the point of view of the $h$-vector or something else.  
So, we ask if the vice versa of Conjecture \ref{conj rnc} is also true. 

\begin{question}
	Let  $X:=\mathbb{S}(\ell_1,\ldots, \ell_r)$ and $Y:=\mathbb{S}(m_1,\ldots, m_s),$ where $r\ge s$, be two star configurations in $\mathbb P^n$ defined by distinct hyperplanes.	Suppose that the $h$-vector of $X\cup Y$ is $$\left(1, \binom{n}{n-1}, \ldots, \binom{r}{n-1}, \binom{s}{n-1}, \ldots, \binom{n}{n-1},1\right).$$
	Then, are $X$ and $Y$ two contact star configurations on a same rational normal curve?
\end{question}

A similar question can be asked in the case of Gorenstein set of points.
\begin{question}\label{q. rnc 2}
	Let  $X:=\mathbb{S}(\ell_1,\ldots, \ell_r)$ and $Y:=\mathbb{S}(m_1,\ldots, m_s),$ where $r\ge s$, be two star configurations in $\mathbb P^n$ defined by distinct hyperplanes. Suppose that $X\cup Y$ is a Gorenstein set of points in $\mathbb P^n$. 
	Then, are $X$ and $Y$ two contact star configurations on a same rational normal curve and either $s=r$ or $s=r-1$?
\end{question}

In the next proposition we positively answer to Question \ref{q. rnc 2} in $\mathbb P^2$ for the case $r=s=3.$
	\begin{proposition}\label{p. viceversa r=s=3}
		Let $X$ and $Y$ be two star configurations, both defined by $3$ distinct lines. Let $X\cup Y$ be a complete intersection of type $(2,3)$. Then $X$ and $Y$ are contact star configurations on the same conic  $\gamma$.
	\end{proposition}
	\begin{proof} Let denote $X=\mathbb S(\ell_1, \ell_2,\ell_3)$ and $Y=\mathbb S(m_1, m_2,m_3)$. Set $P_{ij}:=\ell_i\cap \ell_j$ and $Q_{ij}:=m_i\cap m_j$. 		
		Let denote by $p_{ij}$ and $q_{ij}$ the lines dual to $P_{ij}$ and $Q_{ij}$ and by 
		$L_i$ and $M_j$ the points dual to the lines $\ell_i$ and $m_j$.
		By hypothesis  there is a conic $\mathfrak c$ passing thorough the six points in $X\cup Y$.
		Then, the lines $p_{ij}$ and $q_{ij}$ are tangent to the conic $\mathfrak c^{\vee}$ dual to $\mathfrak c$. Then $\{L_1, L_2, L_3 \}$ and $\{M_1, M_2, M_3 \}$ are $\mathfrak c^{\vee}$-contact star configurations. 
		Hence, from Theorem \ref{t.main} (e), there is a conic $\gamma^{\vee}$ passing through 	$\{L_1, L_2, L_3, M_1, M_2, M_3 \}$.  This proves that $X$ and $Y$ are contact star configurations on a conic $\gamma$, that is, the conic dual to $\gamma^{\vee}.$
	\end{proof}

%
%



\begin{thebibliography}{10}
	
	\bibitem{CoCoA}
	{\sc J.~Abbott, A.~M. Bigatti, and L.~Robbiano}, {\em {CoCoA}: a system for
		doing {C}omputations in {C}ommutative {A}lgebra}.
	\newblock Available at \texttt{http://cocoa.dima.unige.it}.
	
	\bibitem{AS2012}
	{\sc J.-M. Ahn and Y.-S. Shin}, {\em The minimal free resolution of a
		star-configuration in $\mathbb{P}^n$ and the weak lefschetz property},
	Journal of the Korean Mathematical Society, 49 (2012), pp.~405--417.
	
	\bibitem{BCK2016}
	{\sc C.~Bocci, E.~Carlini, and J.~Kileel}, {\em Hadamard products of linear
		spaces}, Journal of Algebra, 448 (2016), pp.~595--617.
	
	\bibitem{BH2010comparing}
	{\sc C.~Bocci and B.~Harbourne}, {\em Comparing powers and symbolic powers of
		ideals}, Journal of Algebraic Geometry, 19 (2010), pp.~399--417.
	
	\bibitem{CCGbipolynomial}
	{\sc E.~Carlini, M.~V. Catalisano, and A.~V. Geramita}, {\em Bipolynomial
		{H}ilbert functions}, Journal of Algebra, 324 (2010), pp.~758--781.
	
	\bibitem{CCGVT2019}
	{\sc E.~Carlini, M.~V. Catalisano, E.~Guardo, and A.~Van~Tuyl}, {\em Hadamard
		star configurations}, Rocky Mountain Journal of Mathematics, 49 (2019),
	pp.~419--432.
	
	\bibitem{CGVT2015}
	{\sc E.~Carlini, E.~Guardo, and A.~Van~Tuyl}, {\em Plane curves containing a
		star configuration}, Journal of Pure and Applied Algebra, 219 (2015),
	pp.~3495--3505.
	
	\bibitem{casey1888sequel}
	{\sc J.~Casey}, {\em A sequel to the First Six Books of the Elements of
		Euclid}, Dublin: Hodges, Figgis, \& Co., 1888.
	
	\bibitem{catalisano1991fat}
	{\sc M.~V. Catalisano}, {\em Fat points on a conic}, Communications in Algebra,
	19 (1991), pp.~2153--2168.
	
	\bibitem{cremona1885}
	{\sc L.~Cremona}, {\em Elements of projective geometry}, Clarendon Press, 1885.
	
	\bibitem{dg1984}
	{\sc E.~Davis and A.~Geramita}, {\em The hilbert function of a special class of
		1-dimensional cohen-macaulay graded algebras}, The curves seminar at
	Queen’s, Queen’s Papers in Pure and Appl. Math, 67 (1984), pp.~1--29.
	
	\bibitem{coordinate-wise}
	{\sc P.~Dey, P.~G{\"o}rlach, and N.~Kaihnsa}, {\em Coordinate-wise powers of
		algebraic varieties}, Beitr{\"a}ge zur Algebra und Geometrie/Contributions to
	Algebra and Geometry,  (2020), pp.~1--33.
	
	\bibitem{fatabbi2001}
	{\sc G.~Fatabbi}, {\em On the resolution of ideals of fat points}, Journal of
	Algebra, 242 (2001), pp.~92--108.
	
	\bibitem{FGM2018}
	{\sc G.~Favacchio, E.~Guardo, and J.~Migliore}, {\em On the arithmetically
		{C}ohen-{M}acaulay property for sets of points in multiprojective spaces},
	Proceedings of the American Mathematical Society, 146 (2018), pp.~2811--2825.
	
	\bibitem{FM2019}
	{\sc G.~Favacchio and J.~Migliore}, {\em Multiprojective spaces and the
		arithmetically {C}ohen--{M}acaulay property}, Mathematical Proceedings of the
	Cambridge Philosophical Society, 166 (2019), pp.~583--597.
	
	\bibitem{GHM2013}
	{\sc A.~V. Geramita, B.~Harbourne, and J.~Migliore}, {\em Star configurations
		in $\mathbb{P}^n$}, Journal of Algebra, 376 (2013), pp.~279--299.
	
	\bibitem{harbourne_1998}
	{\sc B.~Harbourne}, {\em Free resolutions of fat point ideals on
		${{\mathbb{P}}^{2}}$}, Journal of Pure and Applied Algebra, 125 (1998),
	pp.~213--234.
	
	\bibitem{harbourne_2000}
	\leavevmode\vrule height 2pt depth -1.6pt width 23pt, {\em An algorithm for fat
		points on ${{\mathbb{P}}^{2}}$}, Canadian Journal of Mathematics, 52 (2000),
	p.~123–140.
	
	\bibitem{hartshorneAG}
	{\sc R.~Hartshorne}, {\em Algebraic geometry}, vol.~52 of Graduate Texts in
	Mathematic, Springer-Verlag New Yorks, 1977.
	
	\bibitem{kleiman1998}
	{\sc S.~L. Kleiman}, {\em Bertini and his two fundamental theorems}, 1997.
	
	\bibitem{kreuzer2019}
	{\sc M.~Kreuzer, T.~N. Linh, and T.~C. Nguyen}, {\em An application of
		{L}iaison theory to zero-dimensional schemes}, Taiwanese Journal of
	Mathematics,  (2019).
	
	\bibitem{MiglioreBook}
	{\sc J.~Migliore}, {\em Introduction to liaison theory and deficiency modules},
	vol.~165, Birkh\"{a}user, Progress in Mathematics, 1998.
	
	\bibitem{MN3}
	{\sc J.~Migliore and U.~Nagel}, {\em Reduced arithmetically {G}orenstein
		schemes and simplicial polytopes with maximal betti numbers}, Advances in
	Mathematics, 180 (2003), pp.~1--63.
	
	\bibitem{MN20}
	\leavevmode\vrule height 2pt depth -1.6pt width 23pt, {\em Applications of
		liaison}, Expository paper,
	\url{https://www3.nd.edu/~jmiglior/MN14-webpg.pdf},  (2020).
	
	\bibitem{PS2015}
	{\sc J.~P. Park and Y.-S. Shin}, {\em The minimal free graded resolution of a
		star-configuration in $\mathbb{P}^n$}, Journal of Pure and Applied Algebra,
	219 (2015), pp.~2124--2133.
	
	\bibitem{shin2012hilbert}
	{\sc Y.-S. Shin}, {\em On the {H}ilbert function of the union of two linear
		star-configurations in $\mathbb{P}^2$}, Journal of the Chungcheong
	Mathematical Society, 25 (2012), pp.~553--562.
	
	\bibitem{shin2013some}
	{\sc Y.~S. Shin}, {\em Some examples of the union of two linear
		star-configurations in $\mathbb{P}^2$ having generic hilbert function},
	Journal of the Chungcheong Mathematical Society, 26 (2013), pp.~403--409.
	
\end{thebibliography}
\end{document}